\documentclass[10pt]{amsart}
\usepackage{amsthm,amsmath,amssymb,tikz}
\usetikzlibrary{arrows}
\usetikzlibrary{patterns}
\usepackage{mathcomp}

\title[The genus $g$ series for $(p,q,n)$-dipoles]{A finite generating set for the genus $g$ $(p,q,n)$-dipole series from perturbative Yang-Mills theory}
\author{D. M. Jackson$^1$ and C. A. Sloss$^2$}



\theoremstyle{plain}
\newtheorem{thm}{Theorem}[section] 
\newtheorem{lemma}[thm]{Lemma} 
\newtheorem{prop}[thm]{Proposition} 
\newtheorem{cor}[thm]{Corollary}

\newtheorem{defn}[thm]{Definition} 

\theoremstyle{definition} 

\newtheorem{remark}[thm]{Remark}

\newcommand{\bdot}{\text{\textbullet}}
\newcommand{\wdot}{\text{\textopenbullet}}

\newcommand{\rj}{\rho}
\newcommand{\nrj}{\nu}
\newcommand{\dhat}{\widehat{\mathcal{D}}}

\newcommand{\strc}{\eta}

\newcommand{\serA}{A}
\newcommand{\serB}{B}
\newcommand{\serC}{\Gamma}

\newcommand{\localvari}{\zeta}
\newcommand{\localvarii}{\xi}

\thanks{
${\hspace{-1ex}}^1$ Department of Combinatorics and Optimization, University of Waterloo, Waterloo, Ontario, Canada. Partially supported by an NSERC Discovery Grant. \texttt{dmjackson@math.uwaterloo.ca}}

\thanks{
${\hspace{-1ex}}^2$ Department of Combinatorics and Optimization, University of Waterloo, Waterloo, Ontario, Canada. Partially supported by an NSERC Postgraduate Scholarship. \texttt{csloss@theorem.ca}}

\begin{document}

\maketitle

\begin{abstract}

There is an emerging class of permutation factorization questions that cannot be expressed wholly in terms of the centre of the group algebra of the symmetric group. We shall term these \emph{non-central}. A notable instance appears in recent work of Constable \textit{et al.} \cite{ConstableFreedmanHeadrick:2002} in perturbative Yang-Mills theory on the determination of a 2-point correlation function of the Berenstein-Maldacena-Nastase operators by means of Feynman diagrams. In combinatorial terms, this question relates to $(p,q,n)$-\textit{dipoles}:  loopless maps with exactly two vertices and $n$ edges, with two distinguished edges, separated by $p$ edges at one vertex and $q$ edges at the other.

By the introduction of join and cut operators, we construct a formal partial differential equation which uniquely determines a generating series from which the $(p,q,n)$-dipole series may be obtained. 
Moreover, we exhibit a set of functions with the property that the genus $g$ solution to this equation may be obtained recursively as an explicit finite linear combination of these. These functions have explicit expressions as sums indexed by elementary combinatorial objects, and we demonstrate how the recursion can be used to give series solutions for surfaces of low genera.

\end{abstract}

\section{Introduction}

The question of enumerating maps (2-cell embeddings) with $n$ edges in orientable surfaces with respect to vertex- and face-degree type may be studied using Tutte's encoding \cite{Tutte:1984} of a map as a pair $(\epsilon, \nu)$, where $\epsilon$ is a fixed point-free involution, $\nu$ is a permutation whose cycle type is the vertex-degree sequence of the map, and $\epsilon\nu$ is a permutation whose cycle type is the face-degree sequence of the map. These sequences (which are topological invariants of the map) index the conjugacy classes of the symmetric group $\mathfrak{S}_{2n}$, so the question of enumerating these maps may be conducted entirely within the centre $Z(2n)$ of the group algebra $\mathbb{C}[\mathfrak{S}_{2n}]$. This is an example of a \emph{central problem}. However, there are several other important permutation factorization problems involving distinguished substructures which are a barrier to centrality, and such problems are called \emph{non-central}. In this paper, we analyze the non-central \emph{$(p,q,n)$-dipole problem} by means of a join-cut analysis. We introduce a combinatorial refinement of the $(p,q,n)$-dipole problem and prove that the generating series for this refinement is the solution to the partial differential equation given in Theorem~\ref{thm:abcd-PDE}. Series solutions for surfaces of small genus are then derived from this equation. These solutions are expressed in terms of a set of functions which are natural to the problem, and which can be expressed as a sum indexed by compositions of a binary string.

\subsection{Loopless dipoles}

A \emph{dipole} is a 2-cell embedding of a graph with exactly two vertices in a locally orientable surface. In this paper, all surfaces are orientable, all dipoles are loopless, and all dipoles are \emph{rooted} by the selection of an edge and a vertex. In diagrams, the root edge will be denoted by an arrow on the edge, directed away from the root vertex. The selection of a root edge and vertex uniquely identifies a \emph{root face}, namely, the face encountered first on a counterclockwise circulation of the root vertex, starting at the root edge. The corner of the root face which is incident with the root edge and vertex is called the \emph{root corner}. These definitions are illustrated in Figure \ref{fig:roots}. Let $\mathcal{D}$ denote the set of dipoles, and let $\mathcal{D}_n$ denote the set of dipoles having $n$ edges. The following notation will be used for $D\in \mathcal{D}$.
\begin{itemize}
\item[-]
$n(D)$ is the number of edges of $D$.
\item[-]
$m(D)$ is the number of faces of $D$.
\item[-]
$g(D)$ is the genus of the embedding surface. 
\end{itemize}
We note that $g(D) = \frac{1}{2}(n(D)-m(D))$ by the Euler-Poincar\'{e} formula.
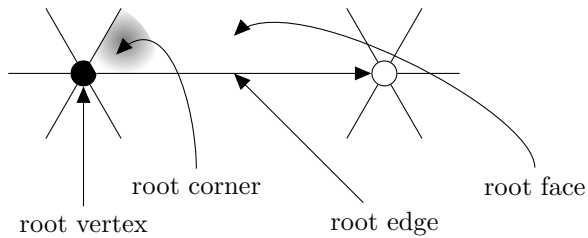
\begin{figure}
\begin{center}
\begin{tikzpicture}[>=triangle 45]
\path (0,0) node[draw,shape=circle, fill=black] (v0) {};
\path (4,0) node[draw,shape=circle,fill=white] (v1){};

\shade [inner color=gray, outer color=white] (v0) -- (1,0) arc (0:60:1cm) -- (v0);

\draw[->] (v0) -- (v1);

\path (v0)+(60:1cm) coordinate (v0a);
\path (v0)+(120:1cm) coordinate (v0b);
\path (v0)+(180:1cm) coordinate (v0c);
\path (v0)+(240:1cm) coordinate (v0d);
\path (v0)+(300:1cm) coordinate (v0e);

\draw (v0) -- (v0a);
\draw (v0) -- (v0b);
\draw (v0) -- (v0c);
\draw (v0) -- (v0d);
\draw (v0) -- (v0e);

\path (v1)+(60:1cm) coordinate (v1a);
\path (v1)+(120:1cm) coordinate (v1b);
\path (v1)+(0:1cm) coordinate (v1c);
\path (v1)+(240:1cm) coordinate (v1d);
\path (v1)+(300:1cm) coordinate (v1e);

\draw (v1) -- (v1a);
\draw (v1) -- (v1b);
\draw (v1) -- (v1c);
\draw (v1) -- (v1d);
\draw (v1) -- (v1e);

\path (0,-2) node (label1) {root vertex};
\draw [->] (label1) -- (v0);

\path (4,-2) node (label2) {root edge};
\draw[->] (label2) -- (2,0);

\path (1.5,-1.5) node (label3) {root corner};
\path (v0)+(30:0.5cm) coordinate (label3end);
\draw[->] (label3) .. controls +(90:1cm) and +(45:1cm) .. (label3end);

\path (6,-1.5) node (label4) {root face};
\path (2,0.5) coordinate (label4end);
\draw[->] (label4) .. controls +(90:1cm) and +(45:1cm) .. (label4end);

\end{tikzpicture}
\end{center}
\caption{The root edge, vertex, face and corner of a loopless dipole.}
\label{fig:roots}
\end{figure}

\subsubsection{$(p,q,n)$-dipoles}
\label{sec:pqn-definition}

A non-central refinement of loopless dipoles, called \emph{$(p,q,n)$-dipoles}, was introduced by Constable \textit{et al.} \cite{ConstableFreedmanHeadrick:2002} in the study of duality between string theory and Yang-Mills theory. It concerns the set $\dhat$ of rooted dipoles with a second distinguished edge, to be called the \emph{secondary edge}, which will be indicated in diagrams by a dashed line. An edge which is neither the root edge nor the secondary edge shall be referred to as an \emph{ordinary edge}.
Given such a dipole, the neighbourhoods of each vertex can be partitioned into four regions as follows:
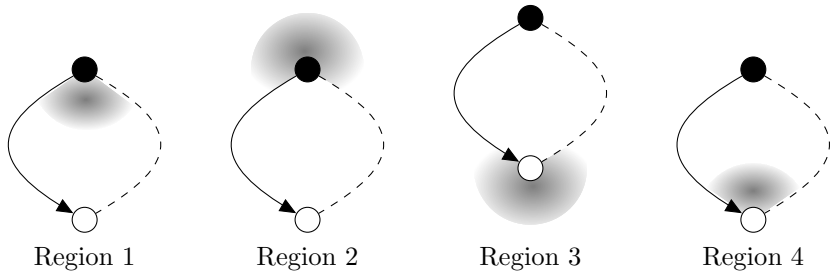
\begin{center}
\begin{tabular}{cccc}
\begin{tikzpicture}[>=triangle 45]

\shade [inner color=gray, outer color=white]  (0,0) -- ++(215:0.75cm) arc (210:325:0.75cm) -- (0,0);

\path (0,0) node[draw,shape=circle, fill=black] (v0) {};
\path (v0)+(270:2cm) node[draw,shape=circle,fill=white] (v1){};

\draw[->] (v0) .. controls +(210:1.5cm) and +(150:1.5cm) .. (v1);
\draw[dashed] (v0) .. controls +(-30:1.5cm) and +(30:1.5cm) .. (v1);


\end{tikzpicture}
& 
\begin{tikzpicture}[>=triangle 45]

\shade [inner color=gray, outer color=white] (0,0) -- ++(210:0.75cm) arc (210:-30:0.75cm) -- (0,0);

\path (0,0) node[draw,shape=circle, fill=black] (v0) {};
\path (v0)+(270:2cm) node[draw,shape=circle,fill=white] (v1){};

\draw[->] (v0) .. controls +(210:1.5cm) and +(150:1.5cm) .. (v1);
\draw[dashed] (v0) .. controls +(-30:1.5cm) and +(30:1.5cm) .. (v1);


\end{tikzpicture} 
&
\begin{tikzpicture}[>=triangle 45]

\shade [inner color=gray, outer color=white] (0,-2) -- ++(30:0.75cm) arc (30:-210:0.75cm) -- (0,-2);

\path (0,0) node[draw,shape=circle, fill=black] (v0) {};
\path (v0)+(270:2cm) node[draw,shape=circle,fill=white] (v1){};

\draw[->] (v0) .. controls +(210:1.5cm) and +(150:1.5cm) .. (v1);
\draw[dashed] (v0) .. controls +(-30:1.5cm) and +(30:1.5cm) .. (v1);


\end{tikzpicture}
&

\begin{tikzpicture}[>=triangle 45]

\shade [inner color=gray, outer color=white] (0,-2) -- ++(35:0.75cm) arc (35:145:0.75cm) -- (0,-2);

\path (0,0) node[draw,shape=circle, fill=black] (v0) {};
\path (v0)+(270:2cm) node[draw,shape=circle,fill=white] (v1){};

\draw[->] (v0) .. controls +(210:1.5cm) and +(150:1.5cm) .. (v1);
\draw[dashed] (v0) .. controls +(-30:1.5cm) and +(30:1.5cm) .. (v1);


\end{tikzpicture}

\\
Region 1 & Region 2 & Region 3 & Region 4
\end{tabular}
\end{center}

\noindent(In these diagrams, the ordinary edges are suppressed.) The partition of the neighbourhoods of the root vertices into these four regions permits the following definition to be made.
\begin{defn}[Root jump and Non-root jump]
Let $D$ be a rooted dipole with a secondary edge. The \emph{root jump} (resp. \emph{non-root jump}) of $D$, denoted by $\rj(D)$ (resp. $\nrj(D)$), is one plus the number of edges intersecting the interior of Region 1 (resp. Region 3). 
\end{defn}
\noindent We are now in a position to define the main problem of this paper.
\begin{defn}[$(p,q,n)$-dipole problem]
A dipole $D$ with $n$ edges for which $\rj(D)=p$ and $\nrj(D)=q$ is referred to as a \emph{$(p,q,n)$-dipole}. 
\label{defn:pqn-dipole}
The \emph{$(p,q,n)$-dipole problem} is the problem of determining the number of $(p,q,n)$-dipoles in an orientable surface of genus $g$. 
\label{defn:pqn-dipole}
\end{defn}

Maps arise in a physical context as embeddings of Feynman diagrams. Constable \textit{et al}. \cite{ConstableFreedmanHeadrick:2002} were concerned with the free two-point functions of the Berenstein-Maldacena-Nastase operators. These functions can be obtained from the $(p,q,n)$-dipole series
\[
\Phi := \sum_{D\in \dhat}r^{\rj(D)}s^{\nrj(D)} t^{n(D)}u^{2g(D)},
\]
where $r$, $s$, $t$ and $u$ are indeterminates. 

Asymptotic expressions for the torus and double torus were given by Constable \textit{et al.} \cite{ConstableFreedmanHeadrick:2002}, and exact expressions for the coefficients of $\Phi$ in these cases were given by Visentin and Wieler~\cite{VisentinWieler:2007}.

\subsection{Centrality and Non-centrality}
\label{sec:centrality-and-non-centrality}

Historically, use of the symmetric group algebra $\mathbb{C}[\mathfrak{S}_n]$ has proven to be an effective approach to map enumeration problems. Although the present paper does not apply this approach to the $(p,q,n)$-dipole problem, it provides a context in which the complications of the $(p,q,n)$-dipole problem may be understood.

Map enumeration problems, such as the loopless dipole problem, can be studied using the centre $Z(n)$ of $\mathbb{C}[\mathfrak{S}_n]$. For maps in general, this is done using Tutte's \cite{Tutte:1984} encoding of a map as a rotation system. Although Tutte's encoding does not exclude loops, Kwak and Lee's \cite{Kwak:1992} adaptation of it does in the case of dipoles, encoding the latter as products of two full cycles. For a partition $\lambda$ of $n$, let $K_{\lambda} = \sum_{\sigma\in \mathcal{C}_{\lambda}} \sigma$, where $\mathcal{C}_{\lambda}$ is the conjugacy class in $\mathfrak{S}_n$ consisting of all permutations of cycle type $\lambda$. The loopless dipole problem is equivalent to determining $(n-1)!^{-1} K_{(n)}^2$ in $Z(n)$, and is therefore \emph{central} since it may be done using the character theory of the symmetric group. The set $\{K_{\lambda}\}_{\lambda\vdash n}$ is a basis for $Z(n)$, and thus the solution to any central problem is necessarily a class function. 

In contrast, the $(p,q,n)$-dipole problem is \emph{non-central} since it cannot be solved by central methods because the weight functions $\rj$ and $\nrj$ are not class functions. Thus, it is an example of a \emph{non-central} problem. Centralizers of $\mathbb{C}[\mathfrak{S}_n]$, such as
\[
Z_k(n) := \{ g \in \mathbb{C}[\mathfrak{S}_n] : \pi g \pi^{-1} = g \text{ for all } \pi \in \mathfrak{S}_{n-k}\}
\]
provide an algebraic context for studying non-central problems, with the non-negative integer $k$ providing a measure of non-centrality. For prescribed values of $p, q,$ and $n$, the set of $(p,q,n)$-dipoles may be encoded in $Z_2(n)$ by the element
\[
\sum_{\substack{\sigma_1\in \mathcal{C}_{(n)} \\ \sigma_1^q(n)=n-1}} \sigma_1 \sum_{\substack{\sigma_2\in \mathcal{C}_{(n)} \\ \sigma_2^p(n)=n-1}} \sigma_2.
\]
Since this does not lie in $Z(n)$, the usual character-theoretic techniques for map enumeration may not be applied. Consequently, in this paper we apply to the $(p,q,n)$-dipole problem a join-cut approach, which is not based on character theory. 

\subsection{Join-cut Analysis}

Join-cut analysis is an algebraic-combinatorial approach, applied to a set $\mathcal{S}$ of combinatorial objects, that requires the following:

\noindent \textbf{(1)}
the characterization of a substructure in $\sigma\in \mathcal{S}$ and parts of $\sigma$ such that the deletion (or, dually, addition) of the substructure either \emph{cuts} some part of $\sigma$ into two parts or \emph{joins} two parts of $\sigma$ into a single part; \\
\noindent \textbf{(2)}
the realization of the combinatorial cut and join operations as weight-preserving differential operators acting on the generating series, which in turn leads to a partial differential equation for the generating series; \\
\noindent \textbf{(3)}
and a means of analyzing the partial differential equation in order to deduce combinatorial information about $\mathcal{S}$. Of course, a full solution to the equation is the ultimate goal.

Join-cut analysis has been applied to many combinatorial problems such as the enumeration of factorizations into transpositions \cite{Goulden:1994} and the enumeration of ramified covers of the sphere (see, for example, \cite{GouldenJacksonVakil:2005}). Notably, these are all examples of problems which may, alternatively, be studied using central methods. The join-cut approach, however, does not explicitly rely on character theory and thus does not require the weight functions being studied to be class functions. Thus, it is a natural approach to adopt in studying the $(p,q,n)$-dipole problem since it is not impeded by the non-centrality of the problem. 

\begin{remark}
Many of these problems contain a connectivity condition which, in the character-based approach, is enforced algebraically by taking the logarithm of the generating series. This results in character sums that are difficult to resolve. On the other hand, the join-cut approach has the additional benefit of enforcing connectivity directly, at the combinatorial level. This is not a concern with the enumeration of loopless dipoles, since connectivity is forced. \hfill $\vert$
\end{remark}

Our approach is a synthesis of two ideas arising in previous applications of join-cut analysis. The first is Kwak and Shim's \cite{KwakShim:2002} argument that adding an edge to a loopless dipole results in either a face being cut into two faces of smaller degree, or two faces being joined into a face of larger degree. Our strategy is to introduce a refined marking of bi-rooted dipoles which allows us to track how $\theta$ and $\vartheta$ change when an edge is added to a bi-rooted dipole. The second idea is the introduction of additional indeterminates to record ``non-central'' information about this marking, in a manner similar to that used by Goulden and Jackson \cite{GouldenJackson:2007} in the enumeration of transitive powers of Jucys-Murphy elements.

\subsection{Organization of the Paper}

In Section \ref{sec:join-cut-operators}, we introduce a refinement of the $(p,q,n)$-dipole problem, called the \emph{$(a,b,c,d)$-dipole problem}, which is amenable to a join-cut approach, and from which the solution to the $(p,q,n)$-dipole problem may be recovered. We give an encoding for $(a,b,c,d)$-dipoles such that the generating series is determined by two partial differential equations (Lemma \ref{lemma:ab00-PDE} and Theorem \ref{thm:abcd-PDE}). Section \ref{sec:genus-recursion} describes a process, recursive in genus, for determining the solution to these equations for a given orientable surface. We identify a family of functions $\mathcal{F}$ such that solutions can be written as a linear combination of elements of $\mathcal{F}$, and give expressions for these functions as a sums indexed by combinatorial objects (Theorems \ref{thm:tau-nonrecursive} and \ref{thm:sigma-nonrecursive}). Section \ref{sec:small-genus-solutions} applies the results of Section \ref{sec:genus-recursion} to give series solutions on surfaces of small genus (Lemma \ref{lemma:ab00-on-torus}, Lemma \ref{lemma:ab00-double-torus}, and \ref{thm:abcd-torus}), and contains the lemmas needed to extract coefficients from these series. 

\section{Join-cut Operators For the $(p,q,n)$-dipole Problem}
\label{sec:join-cut-operators}

\subsection{A refinement of the problem: $(a,b,c,d)$-dipoles}

We consider the effect on the jumps $\rj(D)$ and $\nrj(D)$ of $D$ by the insertion of an edge $e$. This may be characterized by the regions through which $e$ approaches the two vertices, as shown in Figure  \ref{figure:edge-types}. This classifies $e$ into one of four edge types, which are defined in Table \ref{table:edge-types}, together with the increments in $\rj$ and $\nrj$.
\begin{center}
\begin{table}
\begin{tabular}{cccc}
edge type & ends through regions & increment in $\rj$ & increment in $\nrj$ \\
\hline
$a$ & 2, 3 & 0 & 1 \\
$b$ & 1, 3 & 1 & 1 \\
$c$ & 2, 4 & 0 & 0 \\
$d$ & 1, 4 & 1 & 0
\end{tabular}
\caption{Classification of four edge types in a bi-rooted dipole.}
\label{table:edge-types}
\end{table}
\end{center}

Let $\alpha(D),\beta(D),\gamma(D)$ and $\delta(D)$ denote the number of $a$-, $b$-, $c$-, and $d$-edges of $D$, respectively. The jumps of a dipole can be recovered from this information by
\[
\rj(D) = \beta(D)+\delta(D) + 1
\quad \text{ and } \quad
\nrj(D) = \alpha(D)+\beta(D) + 1.
\]
Moreover, the total number of edges in a dipole is
\[
n(D) = \alpha(D)+\beta(D)+\gamma(D)+\delta(D)+2.
\]

\begin{figure}
\begin{center}
\begin{tabular}{cccc}
\begin{tikzpicture}[>=triangle 45]

\path (0,0) node[draw,shape=circle, fill=black] (v0) {};
\path (v0)+(270:2cm) node[draw,shape=circle,fill=white] (v1){};

\draw[->] (v0) .. controls +(195:1.5cm) and +(165:1.5cm) .. (v1);
\draw[dashed] (v0) .. controls +(-15:1cm) and +(15:1cm) .. (v1);
\draw[ultra thick] (v0) .. controls +(0:1.5cm) and +(0:1.5cm) .. (v1);

\end{tikzpicture}
& 
\begin{tikzpicture}[>=triangle 45]

\path (0,0) node[draw,shape=circle, fill=black] (v0) {};
\path (v0)+(270:2cm) node[draw,shape=circle,fill=white] (v1){};

\draw[->] (v0) .. controls +(195:1.5cm) and +(165:1.5cm) .. (v1);
\draw[dashed] (v0) .. controls +(-15:1cm) and +(15:1cm) .. (v1);
\draw[ultra thick] (v0) .. controls +(270:1.5cm) and +(0:1.5cm) .. (v1);

\end{tikzpicture}
&
\begin{tikzpicture}[>=triangle 45]

\path (0,0) node[draw,shape=circle, fill=black] (v0) {};
\path (v0)+(270:2cm) node[draw,shape=circle,fill=white] (v1){};

\draw[->] (v0) .. controls +(195:1.5cm) and +(165:1.5cm) .. (v1);
\draw[dashed] (v0) .. controls +(-15:1cm) and +(15:1cm) .. (v1);
\draw[ultra thick] (v0) .. controls +(0:1.5cm) and +(90:1.5cm) .. (v1);

\end{tikzpicture}&
\begin{tikzpicture}[>=triangle 45]

\path (0,0) node[draw,shape=circle, fill=black] (v0) {};
\path (v0)+(270:2cm) node[draw,shape=circle,fill=white] (v1){};

\draw[->] (v0) .. controls +(195:1.5cm) and +(165:1.5cm) .. (v1);
\draw[dashed] (v0) .. controls +(-15:1cm) and +(15:1cm) .. (v1);
\draw[ultra thick] (v0)--(v1);

\end{tikzpicture} \\
$a$-edge & $b$-edge & $c$-edge & $d$-edge
\end{tabular}
\end{center}
\caption{Classification an ordinary edge by the regions through which it is incident with the two vertices.}
\label{figure:edge-types}
\end{figure}
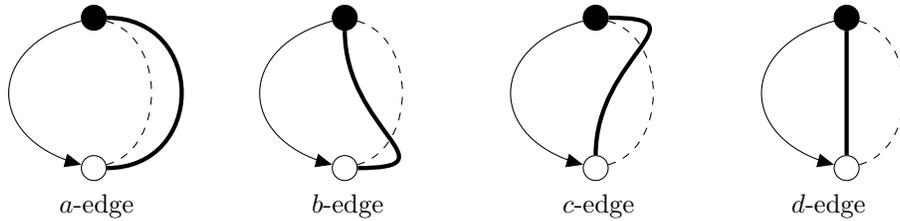

To enumerate dipoles with respect to the number of $a$-, $b$-, $c$-, and $d$-edges, mark the corners of the dipole $D$ in Region 1 with a black dot, $\bdot$, and the corners in Region 2 with a white dot, $\wdot$. Recall, from Section \ref{sec:pqn-definition}, that these are complementary regions whose union is an open neighbourhood of the root vertex. The marking of corners of the faces in this manner plays an important role in Theorem \ref{thm:abcd-PDE}, where it allows us to differentiate between $c$-edges and $d$-edges. This marking associates a unique binary string $R(D)$ to the root face, namely, the string encountered during a counterclockwise boundary walk of the root face, starting at the root corner. Throughout this paper we use the notational convention that for a binary string $R$, the $i^{\text{th}}$ symbol of $R$ is denoted by $R_i$. 

Unlike the root face, the non-root faces do not have a canonical corner to start a boundary walk, so they must be considered up to cyclic equivalence. We therefore define the following.
\begin{defn}[Cyclic binary string]
Given an element $S\in \{\bdot,\wdot\}^*$, let $(S)$ denote the multiset of cyclic shifts of $S$. Such a set shall be referred to as a \emph{cyclic binary string}. The notation $\mathcal{S}(\bdot, \wdot)$ denotes the set of all cyclic binary strings on the symbols $\bdot$ and $\wdot$.
\end{defn}
For example, the cyclic shifts of $\wdot\bdot\wdot\bdot$ are
\[
\wdot\bdot\wdot\bdot, \bdot\wdot\bdot\wdot, \wdot\bdot\wdot\bdot, \text{ and } \bdot\wdot\bdot\wdot, 
\]
and thus 
\[
(\wdot\bdot\wdot\bdot) = (\bdot\wdot\bdot\wdot) = \{ 2 \wdot\bdot\wdot\bdot, 2 \bdot\wdot\bdot\wdot\}.
\]
Each non-root face of a dipole may be associated with the cyclic binary string encountered in a counterclockwise boundary walk of the face. 

Let $\Lambda(D)$ denote the multiset of cyclic binary strings corresponding to the non-root faces of $D$.
We use two infinite classes of indeterminates to record the information associate with this marking: $\{g_S : S\in \{\bdot,\wdot\}^*\}$, indexed by binary strings, and $\{f_{(S)} : (S)\in \mathcal{S}(\bdot,\wdot)\}$, indexed by cyclic binary strings. Let $f_{\Lambda(D)} := \prod_{(S)\in \Lambda(D)} f_{(S)}$. A birooted dipole $D$ is encoded by the monomial $g_{R(D)} f_{\Lambda(D)}$. 

This encoding is illustrated using the toroidal bi-rooted dipole shown in Figure \ref{figure:marking-example}, in which the root face is gray. It is encoded by either of the (equivalent) monomials $g_{\bdot\wdot} f_{(\bdot\wdot)}$ or $g_{\bdot\wdot} f_{(\wdot\bdot)}$. The monomials $g_{\wdot\bdot}f_{(\bdot\wdot)}$ and $g_{\wdot\bdot}f_{(\wdot\bdot)}$ do \emph{not} encode this dipole, since the root edge induces a canonical ``starting point'' for reading the symbols on a boundary walk of the root face.

\begin{figure}
\begin{center}
\begin{tikzpicture}[>=triangle 45,scale=0.75]


\path (0,0) coordinate (NW);
\path (6,0) coordinate (NE);
\path (0,-6) coordinate (SW);
\path (6,-6) coordinate (SE);

\path (3,-1.5) coordinate (rv);
\path (3,-4.5) coordinate (nrv);
\path (3,-3) coordinate (middle);

\path (3,0) coordinate (N);
\path (3,-6) coordinate (S);
\path (0,-3) coordinate (W);
\path (6,-3) coordinate (E);

\path(4,0) coordinate (topside);
\path(0,-2) coordinate (leftside);
\path(4,-6) coordinate (bottomside);
\path(6,-2) coordinate (rightside);

\fill[gray] (rv) -- (E) -- (SE) -- (nrv);
\fill[gray] (nrv) -- (S) -- (SW) -- (W);
\fill[gray] (rv) -- (N) -- (NW);


\begin{scope} [>=to]
\draw[densely dotted, ->] (NW) -- (topside);
\draw[densely dotted] (topside) -- (NE) -- (rightside);
\draw[densely dotted, ->] (SW) -- (bottomside);
\draw[densely dotted, ->>] (SW) -- (leftside);
\draw[densely dotted, ->>] (SE) -- (rightside);
\draw[densely dotted] (bottomside) -- (SE);
\draw[densely dotted] (leftside) -- (NW);
\end{scope}

\path (nrv) ++ (0,0.2) coordinate (arrowend);
\draw[->] (rv) -- (arrowend);

\draw (rv) -- (E);
\draw [dashed] (rv) -- (N);
\draw (rv) -- (NW);

\draw (nrv) -- (W);
\draw [dashed] (nrv) -- (S);
\draw (nrv) -- (SE);

\fill[black] (rv)+(30:0.5cm) circle(0.1cm);
\fill[white] (rv)+(120:0.5cm) circle(0.1cm);
\fill[white] (rv)+(210:0.5cm) circle(0.1cm);
\draw (rv)+(120:0.5cm) circle(0.1cm);
\draw (rv)+(210:0.5cm) circle(0.1cm);
\fill[black] (rv)+(300:0.5cm) circle(0.1cm);

\fill[black] (rv) circle(0.2cm);
\fill[white] (nrv) circle(0.2cm);
\draw (nrv) circle(0.2cm);

\end{tikzpicture}
\end{center}
\caption{Example of a bi-rooted dipole, embedded in the torus, in which the corners incident to the root vertex are marked.}
\label{figure:marking-example}
\end{figure}
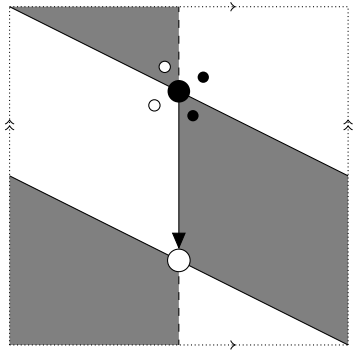

\subsection{The generating series for $(a,b,c,d)$-dipoles}

We may now specify a generating series for $(a,b,c,d)$-dipoles which admits a join-cut analysis. It is a three-stage process in which we introduce classes of edges incrementally to build up classes of dipoles of increasing complexity. 

\noindent \textbf{(1)}
First, enumerate $(a,0,0,0)$-dipoles, which can be regarded as rooted dipoles with a ``doubled'' root edge. This is a central problem. The $(a,0,0,0)$-dipole series $\serA$ is defined by
\[
\serA := g_{\bdot} \sum_{D\in \mathcal{D}} \frac{x^{n(D)}}{n(D)!} u^{2g(D)} f_{\lambda(D)},
\]
where $\lambda(D) = (\wdot^{\lambda_1},\wdot^{\lambda_2},\ldots,\wdot^{\lambda_{m(\lambda)}})$ when $D$ has face-degree sequence $(2\lambda_1,2\lambda_2,\ldots,2\lambda_{m(\lambda)})$, is known. 

\noindent \textbf{(2)}
By considering the adjoining of $b$-edges, obtain a differential equation for the $(a,b,0,0)$-dipole series $\serB'$ defined by
\[
\serB:= \sum_{\substack{D\in \dhat, \\ \gamma(D)=\delta(D)=0}} \frac{x^{\alpha(D)+1}}{(\alpha(D)+1)!} \frac{y^{\beta(D)}}{\beta(D)!}u^{2g(D)} g_{R(D)} f_{\Lambda(D)}.
\]
whose initial condition, at $y=0$, is the $(a,0,0,0)$-dipole series $\serA$.

\noindent \textbf{(3)}
We next consider the addition of edges having an end in Region 4, \textit{i.e.} edges which are either a $c$- or a $d$-edge. The need for a refined marking is now clear, since the symbols $\wdot$ and $\bdot$ specify whether a new edge is a $c$-edge or a $d$-edge. We obtain a differential equation for the $(a,b,c,d)$-dipole series $\serC$ defined by
\[
\serC := \sum_{D\in \dhat} \frac{x^{\alpha(D)+1}}{(\alpha(D)+1)!} \frac{y^{\beta(D)}}{\beta(D)!}\frac{v^{\gamma(D)+\delta(D)}}{(\gamma(D)+\delta(D))!} w^{\delta(D)}u^{2g(D)} g_{R(D)} f_{\Lambda(D)}, 
\]
whose initial condition, at $v=0$, is the $(a,b,0,0)$-dipole series $\serB$.

The $(p,q,n)$-dipole series may be obtained by evaluating the series $\serC$ at $f_{(R)}=1$ and $g_R=1$ for all $R\in \{\bdot,\wdot\}^*$ to ``forget'' the information recorded by the marking of face corners incident with the root vertex, and summing the coefficients corresponding to specified values of $p$ and $q$. We therefore introduce the following notation.
\begin{defn}
For a formal power series $F$ whose coefficients are polynomials in the $f$- and $g$-type indeterminates, let $\langle F\rangle$ denote the series obtained setting $g_R=1$ for all $R \in \{\bdot,\wdot\}^*$, and $f_{(S)}=1$ for all $(S) \in \mathcal{S}(\bdot,\wdot)$.
\end{defn}

\begin{remark}
It is worth noting that, for combinatorial reasons, all series $F$ appearing in this paper satisfy the requirement that $\langle F \rangle$ is well-defined. Indeed, such series are linear in the $g$-type indeterminates, and the sum of degrees of $f$-type indeterminates in a monomial encoding a dipole $D$ is bounded by $\alpha(D)+\beta(D)+\gamma(D)+\delta(D)+1$. \hfill $\vert$
\end{remark}
With the introduction of this notation, the solution to the $(p,q,n)$-dipole problem may be obtained from the solution to the $(a,b,c,d)$-dipole problem as follows.
\begin{prop}
The coefficients of the generating series for $(p,q,n)$-dipoles are given by
\[
[r^ps^qt^nu^{2g}] \Phi = \sum_{0\leq b \leq p-1}  \left[\frac{v^{n-q-1}}{(n-q-1)!}\frac{y^b}{b!} w^{p-1-b}\frac{x^{q-b}}{(q-b)!}u^{2g}\right] \langle \serC \rangle,
\]
where $\serC$ is the $(a,b,c,d)$-dipole series. 
\label{prop:abcd-to-pqn}
\end{prop}
\begin{proof}
Since $\gamma(D) + \delta(D) = n(D)-\nrj(D)-1$, $\rj(D) = \beta(D)+\delta(D)+1$, and $\alpha(D)= \nrj(D)-\beta(D)-1$, the number of $(p,q,n)$-dipoles may be recovered from the $(a,b,c,d)$-dipole series as the given sum of coefficients. 
\end{proof}

\subsection{The initial condition: $(a,0,0,0)$-dipoles}

Before proceeding with this strategy, we first make note of two useful expressions for $\serA$. The first follows directly from the definitions:
\begin{equation}
\serA = g_{\bdot} \sum_{n,m \geq 1} \frac{x^n}{n!} u^{n-m} \!\!\!\!\!\!\!\! \sum_{\substack{\lambda\vdash n \\ \lambda \text{ has } m \text{ parts }}}  \psi_{\lambda} f_{\lambda}. 
\label{eqn:Psi-form-2}
\end{equation}
where $\psi_{\lambda}$ is the number of rooted dipoles with face-degree sequence $2\lambda$. The following explicit expressions for $\psi_{\lambda}$ and $\serA$ are a standard result in the theory of map enumeration.
\begin{lemma}
The number of rooted dipoles with face-degree sequence $2\lambda$ is 
\begin{equation}
\label{eqn:psi-expression}
\psi_{\lambda} = \frac{1}{n} |\mathcal{C}_{\lambda}|(n-1)! \sum_{0\leq k \leq n-1} \binom{n-1}{k}^{-1} [y^k] (1+y)^{-1} \prod_{1\leq i\leq m} (1- (-y)^{\lambda_i}),
\end{equation}
and their generating series is given by
\begin{equation}
\label{eqn:Psi-closed}
\serA = g_{\bdot} L\frac{ e^{\sum_{i\geq 1} \frac{x^iu^{i-1}}{i}f_{(\wdot^i)}(1-(-y)^i)}-1}{1+y},
\end{equation}
where $L$ is the linear transformation defined by
\[
L: f(x,y) \longmapsto \int_0^1s^{-1} f\left(xs,\frac{1-s}{s}\right)ds.
\]
\end{lemma}
\begin{proof}
Using the encoding of rooted dipoles as a pair of full cycle permutations (see Section \ref{sec:centrality-and-non-centrality}), standard results from the character theory of the symmetric group may be used to derive Equation (\ref{eqn:psi-expression}). The beta integral $\binom{n-1}{k}^{-1} = n \int_0^1 s^{n-k-1}(1-s)^k ds$ 
allows us to write $\serA$ in the form given in Equation (\ref{eqn:Psi-closed}).
\end{proof}

\subsection{Join and Cut operators for $(a,b,0,0)$-dipoles}

The following theorem gives a formal partial differential equation for the $(a,b,0,0)$-dipole series $\serB$.
\begin{lemma}
The $(a,b,0,0)$-dipole series $\serB$ is the unique solution to the partial differential equation 
\[
(C' + u^2 J') \serB = \frac{\partial \serB}{\partial y},
\]
where
\begin{align*}
C' &:= \sum_{R\in \{\wdot,\bdot\}^*} \left( \sum_{2\leq i \leq \ell(R)} g_{R_1\cdots R_i} f_{(R_1R_{i+1}\cdots R_{\ell(R)})} \right) \frac{\partial}{\partial g_R}, \\
J' &:=  \sum_{R\in \{\wdot,\bdot\}^*} \sum_{(S) \in \mathcal{S}(\wdot,\bdot)} \left(\sum_{S\in (S)} g_{RR_1 S} \right) \frac{\partial^2}{\partial g_R \partial f_{(S)}},
\end{align*}
and the initial condition is
\[
\serB|_{y=0} = \serA,
\]
where $\serA$ is given by (\ref{eqn:Psi-closed}).
\label{lemma:ab00-PDE}
\end{lemma}
\begin{proof} 
Let $\serB_b:= b![y^b]\serB$. It suffices to show that
\[
\serB_b = (C' + u^2 J')\serB_{b-1}
\]
for $b\geq 1$. The set of all $(a,b,0,0)$-dipoles is generated uniquely from the set of $(a,b-1,0,0)$-dipoles by adding a $b$-edge $e$ in all possible ways such that one end of $e$ is affixed to the root vertex at the root corner of the dipole. Suppose that $D$ is a $(a,b-1,0,0)$-dipole encoded by the monomial
\[
u^{2g} g_R \prod_i f_{(S^{(i)})},
\]
where $R = R_1\ldots R_{\ell(R)}$. We construct operators $C'$ and $J'$ by considering how this monomial changes when $e$ is added to $D$. They will be called the cut and join operators, respectively. There are two cases, since the non-root end of $e$ is added to a corner of either the root face, or a non-root face. 

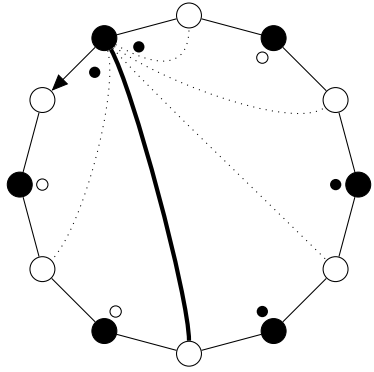
\begin{figure}
\begin{center}
\begin{tikzpicture}[>=triangle 45, scale=0.75]

\path (0,0) coordinate (origin);

\path (origin)+(0:3cm) node[draw,shape=circle,fill=black](v9){};
\path (origin)+(30:3cm) node[draw,shape=circle,fill=white](v10){};
\path (origin)+(60:3cm) node[draw,shape=circle,fill=black](v11){};
\path (origin)+(90:3cm) node[draw,shape=circle,fill=white](v12){};
\path (origin)+(120:3cm) node[draw,shape=circle,fill=black](v1){};
\path (origin)+(150:3cm) node[draw,shape=circle,fill=white](v2){};
\path (origin)+(180:3cm) node[draw,shape=circle,fill=black](v3){};
\path (origin)+(210:3cm) node[draw,shape=circle,fill=white](v4){};
\path (origin)+(240:3cm) node[draw,shape=circle,fill=black](v5){};
\path (origin)+(270:3cm) node[draw,shape=circle,fill=white](v6){};
\path (origin)+(300:3cm) node[draw,shape=circle,fill=black](v7){};
\path (origin)+(330:3cm) node[draw,shape=circle,fill=white](v8){};

\path (origin)+(0:2.6cm) coordinate (v9label){};
\path (origin)+(60:2.6cm) coordinate(v11label){};
\path (origin)+(110:2.6cm) coordinate(v1label-a){};
\path (origin)+(130:2.6cm) coordinate(v1label-b){};
\path (origin)+(180:2.6cm) coordinate(v3label){};
\path (origin)+(240:2.6cm) coordinate(v5label){};
\path (origin)+(300:2.6cm) coordinate(v7label){};

\fill[black] (v1label-a) circle(0.1cm);
\fill[black] (v1label-b) circle(0.1cm);
\draw (v3label) circle(0.1cm);
\draw (v5label) circle(0.1cm);
\fill[black] (v7label) circle(0.1cm);
\fill[black] (v9label) circle(0.1cm);
\draw (v11label) circle(0.1cm);

\draw [->] (v1)--(v2);
\draw (v2)--(v3)--(v4)--(v5)--(v6)--(v7)--(v8)--(v9)--(v10)--(v11)--(v12)--(v1);

\draw[ultra thick] (v1) .. controls +(-60:1cm) and +(90:1cm) .. (v6);

\draw[dotted] (v1) .. controls +(-75:1cm) and +(45:1cm) .. (v4);
\draw[dotted] (v1) .. controls +(-50:1cm) and +(135:1cm) .. (v8);
\draw[dotted] (v1) .. controls +(-40:1cm) and +(215:1cm) .. (v10);
\draw[dotted] (v1) .. controls +(-30:1cm) and +(270:1cm) .. (v12);

\end{tikzpicture}
\end{center}
\caption{A new $b$-edge, indicated by the thickened edge, is added to an $(a,b,0,0)$-dipole in a way that cuts the root face.}
\label{fig:ab00-cut-illustration}
\end{figure}

\noindent \textbf{``Cut'' case}: If the non-root end of $e$ is attached to a corner of the root face then the root face is cut into a root face and a non-root face, and the genus is unchanged. This case is illustrated in Figure \ref{fig:ab00-cut-illustration}. In this example, a root face marked by the indeterminate $g_{\bdot\wdot\wdot\bdot\bdot\wdot}$ is cut into two faces, one of which is a new root face marked by $g_{\bdot\wdot\wdot}$, and the other of which is a non-root face marked by $f_{(\bdot\bdot\wdot\bdot)}$. The dotted lines indicate other permissible choices for the new $b$-edge. The edge $e$ also divides the root corner into two corners each of  which is marked by the symbol $R_1$, one being the new root corner, and the other a new non-root corner. 

To determine the monomial encoding the resulting $(a,b,0,0)$-dipole, consider the symbols encountered on a counterclockwise boundary tour of the root face starting at the root corner. The non-root end of $e$ was added to a corner immediately following $R_i$ for some $2\leq i\leq \ell(R)$. (It cannot be added to the corner following $R_1$, since otherwise it would be a $d$-edge.) Thus, on a boundary tour of the root face, once the symbol $R_i$ is encountered, the edge $e$ completes the boundary tour, so the root face of the resulting dipole is encoded by $g_{R_1\cdots R_i}$. 

To determine the encoding of the new non-root face, consider a counterclockwise boundary tour starting at the newly created non-root corner, marked by $R_1$. The first edge encountered is $e$, and the next symbol encountered will be $R_{i+1}$, after which the boundary tour visits the remaining corners of the former root face, ending when the newly created non-root corner is reached. Hence, this face is encoded by $f_{(R_1R_i\cdots R_{\ell(R)})}$. The differential operator corresponding to replacing $g_R$ with $g_{R_1\cdots R_i}f_{(R_1R_i\cdots R_{\ell(R)})}$ for some $2\leq i \leq \ell(R)$ is the cut operator
\[
C' = \sum_{R\in \{\wdot,\bdot\}^*} \left( \sum_{2\leq i \leq \ell(R)} g_{R_1\cdots R_i} f_{(R_1R_{i+1}\cdots R_{\ell(R)})} \right) \frac{\partial}{\partial g_R}. 
\]

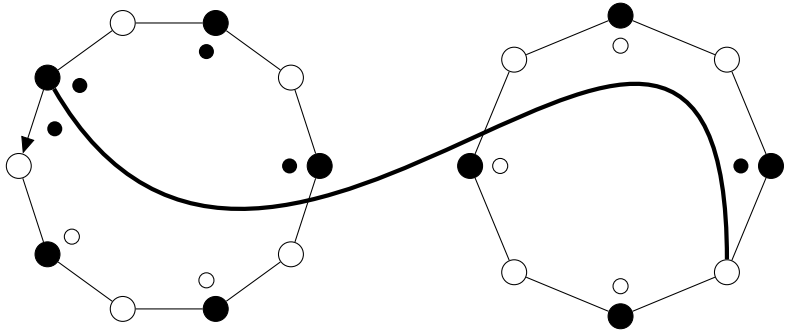
\begin{figure}
\begin{center}
\begin{tikzpicture}[>=triangle 45]

\path (0,0) coordinate (origin);

\path (origin)+(0:2cm) node[draw,shape=circle,fill=black](v7){};
\path (origin)+(36:2cm) node[draw,shape=circle,fill=white](v8){};
\path (origin)+(72:2cm) node[draw,shape=circle,fill=black](v9){};
\path (origin)+(108:2cm) node[draw,shape=circle,fill=white](v10){};
\path (origin)+(144:2cm) node[draw,shape=circle,fill=black](v1){};
\path (origin)+(180:2cm) node[draw,shape=circle,fill=white](v2){};
\path (origin)+(216:2cm) node[draw,shape=circle,fill=black](v3){};
\path (origin)+(252:2cm) node[draw,shape=circle,fill=white](v4){};
\path (origin)+(288:2cm) node[draw,shape=circle,fill=black](v5){};
\path (origin)+(324:2cm) node[draw,shape=circle,fill=white](v6){};

\path (origin)+(0:1.6cm) coordinate (v7label){};
\path (origin)+(72:1.6cm) coordinate(v9label){};
\path (origin)+(162:1.6cm) coordinate(v1label-a){};
\path (origin)+(138:1.6cm) coordinate(v1label-b){};
\path (origin)+(216:1.6cm) coordinate(v3label){};
\path (origin)+(288:1.6cm) coordinate(v5label){};

\fill[black] (v1label-a) circle(0.1cm);
\fill[black] (v1label-b) circle(0.1cm);
\draw (v3label) circle(0.1cm);
\draw (v5label) circle(0.1cm);
\fill[black] (v7label) circle(0.1cm);
\fill[black] (v9label) circle(0.1cm);

\draw[->] (v1)--(v2);
\draw (v2)--(v3)--(v4)--(v5)--(v6)--(v7)--(v8)--(v9)--(v10)--(v1);

\path (6,0) coordinate (origin2);

\path (origin2)+(0:2cm) node[draw,shape=circle,fill=black](w7){};
\path (origin2)+(45:2cm) node[draw,shape=circle,fill=white](w8){};
\path (origin2)+(90:2cm) node[draw,shape=circle,fill=black](w1){};
\path (origin2)+(135:2cm) node[draw,shape=circle,fill=white](w2){};
\path (origin2)+(180:2cm) node[draw,shape=circle,fill=black](w3){};
\path (origin2)+(225:2cm) node[draw,shape=circle,fill=white](w4){};
\path (origin2)+(270:2cm) node[draw,shape=circle,fill=black](w5){};
\path (origin2)+(315:2cm) node[draw,shape=circle,fill=white](w6){};

\path (origin2)+(0:1.6cm) coordinate (w7label){};
\path (origin2)+(90:1.6cm) coordinate (w1label){};
\path (origin2)+(180:1.6cm) coordinate (w3label){};
\path (origin2)+(270:1.6cm) coordinate (w5label){};

\draw (w1) -- (w2) -- (w3) -- (w4) -- (w5) -- (w6) -- (w7) -- (w8) -- (w1);

\draw (w1label) circle(0.1cm);
\draw (w3label) circle(0.1cm);
\draw (w5label) circle(0.1cm);
\fill[black] (w7label) circle(0.1cm);

\draw[ultra thick] (v1) .. controls +(-60:6cm) and +(90:6.5cm) .. (w6);

\end{tikzpicture}
\end{center}
\caption{A new $b$-edge, indicated by the thickened line, is added to an $(a,b,0,0)$-dipole such that the root face is joined to a non-root face.}
\label{fig:ab00-join-illustrated}
\end{figure}

\noindent \textbf{``Join'' case}: Suppose, as illustrated in Figure \ref{fig:ab00-join-illustrated}, that the non-root end of $e$ is attached to a corner of a non-root face, marked by $f_{(S)}$, where $S = S_1\cdots S_{\ell(S)}$ is some fixed representative element of $(S)$.  (For any given $(S)$, the number of choices for a face with $(S)$ as the sequence of symbols encountered on a counterclockwise boundary walk is equal to the degree of $f_{(S)}$ in  $\prod_i f_{(S^{(i)})}$.) To join $e$ to corners of two different faces, it is necessary to add a handle to the surface, thereby increasing its genus by $1$. Adding the edge $e$ joins the two faces marked by $g_R$ and $f_{(S)}$ into a root face of higher degree. In the example illustrated, the root face, marked by $g_{\bdot\wdot\wdot\bdot\bdot}$, is joined to a non-root face, marked by $f_{(\wdot\wdot\wdot\bdot)}$. The resulting face is marked by $g_{\bdot\wdot\wdot\bdot\bdot\bdot\bdot\wdot\wdot\wdot}$. 

Consider a counterclockwise boundary tour of this new face, starting at the root corner. Since the non-root end of $e$ was added to a non-root face, this tour will first visit corners marked by $R_1,R_2,\ldots, R_{\ell(R)}$ before visiting the newly created non-root corner, marked by $R_1$, followed by the edge $e$. The boundary walk then continues around the face marked by $f_{(S)}$, with the sequence of symbols encountered given by some element of the set $(S)$. In terms of the fixed representative $S$, the sequence is $S_i,\cdots,S_{\ell(S)} S_1\cdots S_{i-1}$ for some $1\leq i\leq \ell(S)$, followed by the edge $e$ which returns to the root corner, ending the boundary tour. Thus, the new root face is encoded by $g_{R R_1 S_i,\cdots,S_{\ell(S)} S_1\cdots S_{i-1}}$. 
Summing over all cyclic shifts of $S$ gives the following join operator corresponding to this case.
\[
J' =  \sum_{R\in \{\wdot,\bdot\}^*} \sum_{(S) \in \mathcal{S}(\wdot,\bdot)} \left(\sum_{S\in (S)} g_{RR_1 S} \right) \frac{\partial^2}{\partial g_R \partial f_{(S)}}.
\]
Since contracting the root face of a $(a,0,0,0)$-dipole results in an ordinary rooted dipole with $a(D)+1$ edges, the initial condition is
\begin{align*}
\serB_{y=0} &=\!\!\!\!\!\!\!\! \sum_{\substack{ D \in \dhat, \\b(D)=c(D)=d(D)=0}}\!\!\!\!\!\!\!\! \frac{x^{a(D)+1}}{(a(D)+1)!} u^{2g(D)} g_{\bdot} f_{\Lambda(D)} = g_{\bdot} \sum_{D\in \mathcal{D}} \frac{x^{n(D)}}{n(D)!} u^{2g(D)} f_{\lambda'(D)}.
\end{align*}
This completes the proof.
\end{proof}
It is worth noting that the set of $(a,b,0,0)$-dipoles corresponds to the set of $(p,n-1,n)$-dipoles. While the marking of the faces with black dots and white dots is needed in order to produce a generating series which may be used as an initial condition for the $(a,b,c,d)$-dipole problem, if we record only the degrees of the faces, an argument similar to the proof of Lemma \ref{lemma:ab00-PDE} gives the cut operator
\[
\tilde{C}= \sum_{i\geq 2} \sum_{1\leq j\leq i-1} g_{j+1}f_{i-j} \frac{\partial}{\partial g_i} = \sum_{i\geq 1} \sum_{j\geq 1} g_{i+1}f_j \frac{\partial}{\partial g_{i+j}}
\]
and the join operator 
\[
\tilde{J} = \sum_{i\geq 1}\sum_{j\geq 1} j g_{i+j+1}\frac{\partial^2}{\partial g_i \partial f_j}.
\]
These operators only involve the face-degree sequence and the length of the root face, suggesting that the problem of enumerating $(p,n-1,n)$-dipoles lies in $Z_1(n)$, as opposed to $Z_2(n)$. This is significant because $Z_1(n)$ is a commutative algebra for which an explicit basis of orthogonal idempotents is known, whereas $Z_2(n)$ is non-commutative. The $(p,n-1,n)$-dipole problem is an example of a \emph{near-central} problem that can be solved using algebraic methods in the same spirit as the character-theoretic approach to loopless dipoles. The  reader is directed to \cite{JacksonSloss:2011}, where Strahov's \cite{Strahov:2007} generalized characters are used to solve this problem.

\subsection{Join and Cut operators for $(a,b,c,d)$-dipoles}

A similar analysis gives the following equation for $\serC$.
\begin{thm}
The $(a,b,c,d)$-dipole series $\serC$ is the unique solution to the formal partial differential equation
\begin{equation}
\label{eqn:abcd-PDE}
(C''+u^2 J'') \serC = \frac{\partial \serC}{\partial v}
\end{equation}
where
\begin{align*}
C'' &:= \sum_{R\in \{\wdot,\bdot\}^*} \left( \sum_{2\leq i \leq \ell(R)}w^{\delta_{R_i,\bdot} } g_{R_1R_i\cdots R_{\ell(R)}} f_{(R_2\cdots R_i)} + w g_{R_1}f_{(R)} \right) \frac{\partial}{\partial g_R},\\
J'' &:=  \sum_{R\in \{\wdot,\bdot\}^*} \sum_{(S) \in \mathcal{S}(\wdot,\bdot)}\left( \sum_{S_1S_2\cdots S_{\ell(S)}\in (S)} \!\!\!\!\!\!\!\! w^{\delta_{S_1,\bdot}} g_{R_1S_1\cdots S_{\ell(S)} S_1 R_2\cdots R_{\ell(R)}} \right) \frac{\partial^2}{\partial g_R \partial f_{(S)}},
\end{align*}
and the initial condition is $\serC|_{v=0} = \serB$,
where $\serB$ is given by Lemma \ref{lemma:ab00-PDE}.
\label{thm:abcd-PDE}
\end{thm}
\begin{proof}
Let $\serC_m := m! [v^m]\serC$ be the $(a,b,c,d)$-dipole series in which $c+d=m$. In other words, this is the series for dipoles in which there are $m$ edges having an end in Region~4. The set of dipoles having $m$ edges in Region~4 is uniquely generated by adding an edge $e$, with one end in Region~4, to a dipole $D$ having $m-1$ edges in Region~4 in a canonical way. This edge will be added so that its non-root end is added to the corner which is the first one encountered as one travels clockwise around the non-root vertex starting from the root edge. Placing the new edge in this manner ensures that it will be part of the root face, being the next edge encountered after the root edge on a counterclockwise boundary tour of the root face. The root end of $e$ will be added to a corner which is marked either with a $\bdot$ or with a $\wdot$. If it is added to a corner marked with a $\bdot$, then it is a $d$-edge. If it is added to a corner marked with a $\wdot$, then it is a $c$-edge. Suppose that $D$ is encoded by the monomial 
\[
u^{2g} g_R \prod_i f_{(S^{(i)})},
\]
where $R=R_1\cdots R_{\ell(R)}$. As before, there are two cases depending on whether the root end of $e$ is added to a corner of the root face, or of a non-root face. 

\begin{figure}
\begin{center}
\begin{tikzpicture}[>=triangle 45, scale=0.75]

\path (0,0) coordinate (origin);

\path (origin)+(0:3cm) node[draw,shape=circle,fill=black](v9){};
\path (origin)+(30:3cm) node[draw,shape=circle,fill=white](v10){};
\path (origin)+(60:3cm) node[draw,shape=circle,fill=black](v11){};
\path (origin)+(90:3cm) node[draw,shape=circle,fill=white](v12){};
\path (origin)+(120:3cm) node[draw,shape=circle,fill=black](v1){};
\path (origin)+(150:3cm) node[draw,shape=circle,fill=white](v2){};
\path (origin)+(180:3cm) node[draw,shape=circle,fill=black](v3){};
\path (origin)+(210:3cm) node[draw,shape=circle,fill=white](v4){};
\path (origin)+(240:3cm) node[draw,shape=circle,fill=black](v5){};
\path (origin)+(270:3cm) node[draw,shape=circle,fill=white](v6){};
\path (origin)+(300:3cm) node[draw,shape=circle,fill=black](v7){};
\path (origin)+(330:3cm) node[draw,shape=circle,fill=white](v8){};

\path (origin)+(0:2.6cm) coordinate (v9label-a){};
\path (origin)+(-15:2.6cm) coordinate (v9label-b){};
\path (origin)+(60:2.6cm) coordinate(v11label){};
\path (origin)+(120:2.6cm) coordinate(v1label){};
\path (origin)+(180:2.6cm) coordinate(v3label){};
\path (origin)+(240:2.6cm) coordinate(v5label){};
\path (origin)+(300:2.6cm) coordinate(v7label){};

\fill[black] (v1label) circle(0.1cm);
\draw (v3label) circle(0.1cm);
\draw (v5label) circle(0.1cm);
\fill[black] (v7label) circle(0.1cm);
\fill[black] (v9label-a) circle(0.1cm);
\fill[black] (v9label-b) circle(0.1cm);
\draw (v11label) circle(0.1cm);

\draw [->] (v1)--(v2);
\draw (v2)--(v3)--(v4)--(v5)--(v6)--(v7)--(v8)--(v9)--(v10)--(v11)--(v12)--(v1);

\draw[ultra thick] (v2) .. controls +(-75:3cm) and +(-135:3cm) .. (v9);

\end{tikzpicture}
\end{center}
\caption{A $d$-edge is added to an $(a,b,c,d)$-dipole in a manner which cuts the root face.}
\label{fig:abcd-cut-illustrated}
\end{figure}
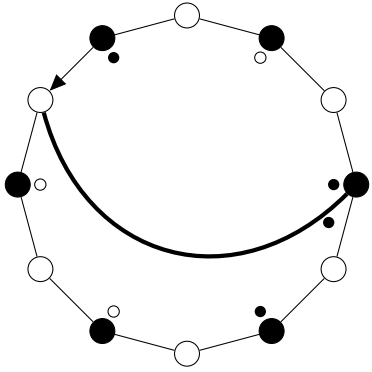

\noindent \textbf{``Cut'' Case}: First, consider the case when the root end of $e$ is added to a corner of the root face, say, the corner marked by $R_i$ where $1\leq i \leq \ell(R)$. This case is illustrated in Figure \ref{fig:abcd-cut-illustrated}. In the ``cut'' case, the root face is cut into two faces, one of which is the new root face, and one of which is a non-root face. In the example illustrated, the root face, originally marked by $g_{\bdot\wdot\wdot\bdot\bdot\wdot}$, is cut into a new root face, marked by $g_{\bdot\bdot\wdot}$, and a non-root face, marked by $f_{(\wdot\wdot\bdot\bdot)}$. This edge is known to be a $d$-edge, as opposed to a $c$-edge, since its root end is added to a corner marked by $\bdot$. Thus, in this example a factor of $w$ must be added to the monomial encoding the dipole.

When $i=1$, the addition of $e$ is a $d$ edge, and creates a new root face which is a digon marked by $R_1$. In this case, the corners of the non-root face are marked by the same sequence of symbols as the old root face, so the contribution in this case is
\[
w g_{R_1} f_{(R)}  \prod_i f_{(S^{(i)})}.
\]

The general case occurs if $2\leq i \leq \ell(R)$. Consider a counterclockwise boundary tour of the new root face, starting at the root corner (marked by $R_1$). After travelling along the root edge, the edge $e$ is encountered, which goes to the corner marked by $R_i$. Continuing the boundary walk, the symbols $R_{i+1},\ldots R_{\ell}$ are encountered, after which the boundary walk is complete. As for the newly-created non-root face, since non-root faces are encoded up to cyclic equivalence any starting point for the boundary walk may be chosen. It is convenient to start at the corner counterclockwise from the non-root end of $e$. Starting from this corner, the corners encountered are marked by $R_2, R_3,\ldots, R_i$, upon which the edge $e$ is encountered, returning to the starting point. Thus, the contribution from this case is
\[
w^{\delta_{R_i,\bdot}} g_{R_1R_i\cdots R_{\ell(R)}} f_{(R_2\cdots R_i)}  \prod_i f_{(S^{(i)})}.
\]
Summing over all cases, the differential operator corresponding to a cut edge is 
\[
C'' = \sum_{R\in \{\wdot,\bdot\}^*} \left( \sum_{2\leq i \leq \ell(R)}w^{\delta_{R_i,\bdot} } g_{R_1R_i\cdots R_{\ell(R)}} f_{(R_2\cdots R_i)} + w g_{R_1}f_{(R)} \right) \frac{\partial}{\partial g_R}.
\]

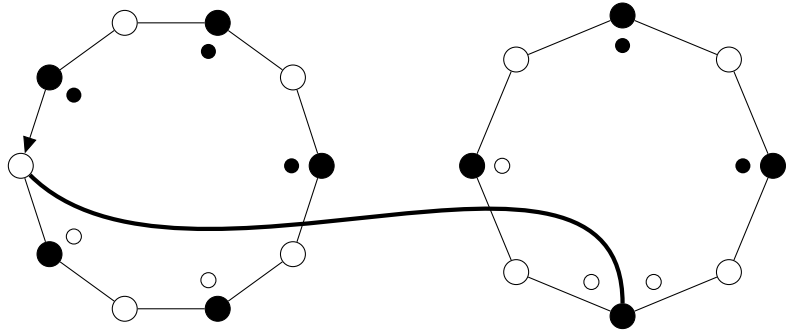
\begin{figure}
\begin{center}
\begin{tikzpicture}[>=triangle 45]

\path (0,0) coordinate (origin);

\path (origin)+(0:2cm) node[draw,shape=circle,fill=black](v7){};
\path (origin)+(36:2cm) node[draw,shape=circle,fill=white](v8){};
\path (origin)+(72:2cm) node[draw,shape=circle,fill=black](v9){};
\path (origin)+(108:2cm) node[draw,shape=circle,fill=white](v10){};
\path (origin)+(144:2cm) node[draw,shape=circle,fill=black](v1){};
\path (origin)+(180:2cm) node[draw,shape=circle,fill=white](v2){};
\path (origin)+(216:2cm) node[draw,shape=circle,fill=black](v3){};
\path (origin)+(252:2cm) node[draw,shape=circle,fill=white](v4){};
\path (origin)+(288:2cm) node[draw,shape=circle,fill=black](v5){};
\path (origin)+(324:2cm) node[draw,shape=circle,fill=white](v6){};

\path (origin)+(0:1.6cm) coordinate (v7label){};
\path (origin)+(72:1.6cm) coordinate(v9label){};
\path (origin)+(144:1.6cm) coordinate(v1label){};
\path (origin)+(216:1.6cm) coordinate(v3label){};
\path (origin)+(288:1.6cm) coordinate(v5label){};

\fill[black] (v1label) circle(0.1cm);
\draw (v3label) circle(0.1cm);
\draw (v5label) circle(0.1cm);
\fill[black] (v7label) circle(0.1cm);
\fill[black] (v9label) circle(0.1cm);

\draw[->] (v1)--(v2);
\draw (v2)--(v3)--(v4)--(v5)--(v6)--(v7)--(v8)--(v9)--(v10)--(v1);

\path (6,0) coordinate (origin2);

\path (origin2)+(0:2cm) node[draw,shape=circle,fill=black](w7){};
\path (origin2)+(45:2cm) node[draw,shape=circle,fill=white](w8){};
\path (origin2)+(90:2cm) node[draw,shape=circle,fill=black](w1){};
\path (origin2)+(135:2cm) node[draw,shape=circle,fill=white](w2){};
\path (origin2)+(180:2cm) node[draw,shape=circle,fill=black](w3){};
\path (origin2)+(225:2cm) node[draw,shape=circle,fill=white](w4){};
\path (origin2)+(270:2cm) node[draw,shape=circle,fill=black](w5){};
\path (origin2)+(315:2cm) node[draw,shape=circle,fill=white](w6){};

\path (origin2)+(0:1.6cm) coordinate (w7label){};
\path (origin2)+(90:1.6cm) coordinate (w1label){};
\path (origin2)+(180:1.6cm) coordinate (w3label){};
\path (origin2)+(285:1.6cm) coordinate (w5label-a){};
\path (origin2)+(255:1.6cm) coordinate (w5label-b){};

\draw (w1) -- (w2) -- (w3) -- (w4) -- (w5) -- (w6) -- (w7) -- (w8) -- (w1);

\fill[black] (w1label) circle(0.1cm);
\draw (w3label) circle(0.1cm);
\draw (w5label-a) circle(0.1cm);
\draw (w5label-b) circle(0.1cm);
\fill[black] (w7label) circle(0.1cm);

\draw[ultra thick] (v2) .. controls +(-45:3cm) and +(90:3cm) .. (w5);

\end{tikzpicture}
\end{center}

\caption{A $c$-edge is added to an $(a,b,c,d)$-dipole such that the root face is joined to a non-root face.}

\label{fig:abcd:join-illustrated}
\end{figure}

\noindent \textbf{``Join'' case}: Next, consider the case when the root end of $e$ is added to a corner of a non-root face, as illustrated in Figure \ref{fig:abcd:join-illustrated}. In this example, the two original faces were marked by $g_{\bdot\wdot\wdot\bdot\bdot}$ and $f_{(\wdot\bdot\bdot\wdot)}$, and the resulting face is marked by $g_{\bdot\wdot\bdot\bdot\wdot\wdot \wdot\wdot\bdot\bdot}$. The edge added in this example is known to be a $c$-edge, as opposed to a $d$-edge, since its root end was added to a corner marked by $\wdot$, so no additional factor of $w$ is contributed.

Suppose the non-root face to which the edge $e$ is added is marked by $f_{(S)}$. (The number of choices for a given $(S)$ is equal to the degree of $f_{(S)}$ in $\prod_i f_{(S^{(i)})}$.) The choice of corner of the non-root face picks out one specific string from the set of cyclic binary strings $(S)$, say, the string $S_1\cdots S_{\ell(S)}$ such that $S_1$ is the corner to which $e$ is added. Then, on a counterclockwise boundary tour of the new root face, the corner labels are encountered in the order
\[
R_1, S_1, S_2,\ldots, S_{\ell(S)}, S_1, R_2, R_3,\ldots, R_{\ell(R)}.
\]
If $S_1= \bdot$, $e$ is a $d$-edge and an additional power of $w$ is needed. If $S_1=\wdot$, $e$ is a $c$-edge and the power of $w$ remains the same. Summing over all choices of a cyclic binary string in $S$, the differential operator corresponding to this case is the join operator
\[
J'' :=  \sum_{R\in \{\wdot,\bdot\}^*} \sum_{(S) \in \mathcal{S}(\wdot,\bdot)}\left( \sum_{S_1S_2\cdots S_{\ell(S)}\in (S)} w^{\delta_{S_1,\bdot}} g_{R_1S_1\cdots S_{\ell(S)} S_1 R_2\cdots R_{\ell(R)}} \right) \frac{\partial^2}{\partial g_R \partial f_{(S)}}.
\]
This completes the proof.
\end{proof}

\section{A Genus-recursive Approach to Solving the $(a,b,c,d)$-dipole problem}
\label{sec:genus-recursion}

\subsection{An overview of the strategy}
\label{sec:strategy}
In this section, we describe a recursive process for determining the solution $\serC^{(g)}:= [u^{2g}]\serC$ for the $(a,b,c,d)$-dipole series for a surface of genus $g$. We let $\serB^{(g)}:= [u^{2g}]\serB$ denote the genus $g$ solution for $(a,b,0,0)$-dipoles. 

\subsubsection{Absorbing $J''$ through dehomogenization} 

When solving Equation (\ref{eqn:abcd-PDE}) recursively in genus, the operator $J''$ may be effectively removed from the equation, at the cost of introducing inhomogeneous terms. Applying $[u^{2g}]$ to both sides of Equation (\ref{eqn:abcd-PDE}), 
\begin{equation}
\label{eqn:abcd-recursion}
\left( \frac{\partial}{\partial v} -C''\right) \serC^{(g)} = J'' \serC^{(g-1)}
\end{equation}
when $g\geq 1$, with
\begin{align*}
\serC^{(0)} &= \sum_{a\geq 0} \sum_{d\geq 0} \frac{x^{a+1}}{(a+1)!} \frac{(vw)^d}{d!} g_{\bdot} f_{(\bdot)}^d f_{(\wdot)}^{a+1}.
\end{align*}
Note that Equation (\ref{eqn:abcd-recursion}) only requires that we be able to apply $J''$ to $\serC^{(g-1)}$, a series which we regard as known, which is a simpler task than solving a differential equation involving $J''$. The analysis of Equation (\ref{eqn:abcd-recursion}) is further simplified by converting it to a differential equation whose initial condition is zero. Let 
\[
\hat{\serC}^{(g)} = \serC^{(g)} - \serC^{(g)}|_{v=0} = \serC^{(g)} - \serB^{(g)}.
\]
Then $\hat{\serC}^{(g)}$ is the solution to the equation
\begin{equation}
\label{eqn:psi-doubleprime-no-ic}
\left( \frac{\partial}{\partial v} -C''\right)  \hat{\serC}^{(g)} = J'' \serC^{(g-1)} + C''\serB^{(g)}.
\end{equation}

Because $\frac{\partial}{\partial v} -C''$ is linear, and all $f$-type indeterminates are constant with respect to it, in order to determine $\hat{\serC}$ it suffices to solve equations of the form
\[
\left( \frac{\partial}{\partial v} -C''\right) H = g_R \omega(v)
\]
for an arbitrary function $\omega$. That is, we need to determine the preimages of $g_R \omega(v)$ under $\left( \frac{\partial}{\partial v} -C''\right)$.

\subsubsection{Constructing candidate $H$ and $\omega$}

Since the series of combinatorial interest are linear in the $g$-type indeterminates, our solution technique is motivated by looking for solutions of the form 
\[
H = g_R \hat{\omega}(v) + F
\]
where $F$ and $\hat\omega$ are some functions to be determined. For solutions of this form,
\begin{align*}
\left( \frac{\partial}{\partial v} -C''\right) H &= g_R \frac{\partial \hat{\omega}}{\partial v} - w^{\delta_{R_2,\bdot}} g_R f_{(R_2)}\hat\omega(v) - \text{``lower order terms''} + \left( \frac{\partial}{\partial v} -C''\right)F,
\end{align*}
where the ``lower order terms'' involve $g$-type indeterminates of length strictly less than that of $R$. Thus, we may inductively choose $F$ so that $\left( \frac{\partial}{\partial v} -C''\right)F$ cancels the ``lower order terms,'' giving
\[
\left( \frac{\partial}{\partial v} -C''\right) H = g_R \left(\frac{\partial \hat{\omega}}{\partial v} - w^{\delta_{R_2,\bdot}} f_{(R_2)}\hat\omega(v)\right).
\]
Making the supposition that $\hat{\omega}$ may be chosen to be the solution to the ordinary differential equations
\begin{equation}
\frac{\partial \hat{\omega}}{\partial v} - w f_{(\bdot)}\hat\omega(v) = \omega(v)
\label{eqn:omega-ODE-1}
\end{equation}
when $R_2=\bdot$, and
\begin{equation}
\frac{\partial \hat{\omega}}{\partial v} - f_{(\wdot)}\hat\omega(v) = \omega(v)
\label{eqn:omega-ODE-2}
\end{equation}
when $R_2=\wdot$, then $\left( \frac{\partial}{\partial v} -C''\right) H = g_R \omega(v)$. This characterization of $\hat\omega$ suggests that the we require a family of functions of $v$ which are naturally indexed by binary strings. Our approach to determining the preimage of $\left( \frac{\partial}{\partial v} -C''\right)$ consists of two steps: \textbf{(Step 1)} the identification of a class of functions which satisfies Equations (\ref{eqn:omega-ODE-1}) and (\ref{eqn:omega-ODE-2}), and \textbf{(Step 2)} the determination of expressions for the preimages with respect to $\frac{\partial}{\partial v} -C''$ and $\frac{\partial}{\partial y} -C'$. Since our solution technique is recursive, a vital property of our descriptions for the preimages is that they are expressed in terms of the class of functions identified in Step 1. Step 1 is carried out in Section \ref{sec:phi-functions}. Step 2 is carried out in Section \ref{sec:ab00-preimages} (for the $(a,b,0,0)$-dipole problem) and Section \ref{sec:abcd-preimages} (for the $(a,b,c,d)$-dipole problem).

\subsubsection{The $(a,b,0,0)$ case}
In a similar manner, the genus $g$ solution to the $(a,b,0,0)$-dipole problem may be determined by solving the equation
\begin{equation}
\label{eqn:ab00-diff-recursion}
\left(\frac{\partial}{\partial y} - C'\right) \serB^{(g)} = J \serB^{(g-1)}
\end{equation}
when $g\geq 1$. When $g=0$,
\[
\serB^{(0)} = g_{\bdot} \sum_{a\geq 0} \frac{x^{a+1}}{(a+1)!} f_{(\wdot)}^{a+1} = g_{\bdot}(\exp(xf_{(\wdot)}) -1).
\]
Let $\hat{\serB}^{(g)} = \serB^{(g)} - \serB^{(g)}|_{y=0}$. When $y=0$, the series $\serB$ corresponds to dipoles with no $b$-edges, for which the root face is a digon. Thus, $C' \serB^{(g)}|_{y=0} =0$, so solving Equation (\ref{eqn:ab00-diff-recursion}) is equivalent to solving
\begin{equation}
\label{eqn:ab00-diff-recursion-2}
\left(\frac{\partial}{\partial y} - C'\right) \hat{\serB}^{(g)} = J \serB^{(g-1)}
\end{equation}
with the initial condition $\hat{\serB}^{(g)} =0$. 

\subsection{The functions $\phi_S$ and $\phi_i$}
\label{sec:phi-functions}

We now exhibit an explicit set of functions which satisfy Equations (\ref{eqn:psi-doubleprime-no-ic}) and (\ref{eqn:ab00-diff-recursion-2}).
\begin{defn}[General $\phi$-function]
Let $i\geq 1$, and let $v,x_1,\ldots,x_i$ be indeterminates. Define
\[
\phi(v,x_1,\ldots,x_i) = 
\begin{cases}
\sum_{n\geq i} h_{n-i}(x_1,\ldots,x_i) \frac{v^n}{n!} & \text{ if } i>0,\\
1 & \text{ if } i=0,
\end{cases}
\]
where $h_j(x_1,\ldots,x_k)$ is the complete symmetric function of total degree $j$ in $k$ indeterminates.
\label{defn:general-phi} 
\end{defn}
\noindent (If $j<0$, the convention $h_j(x_1,\ldots x_k)=0$ is used.) The usefulness of these functions lies in the fact that they satisfy the following essential property, which generalizes Equations (\ref{eqn:omega-ODE-1}) and (\ref{eqn:omega-ODE-2}).
\begin{lemma}
\label{lem:fundamental-phi-property}
For $i\geq 1$, 
\[
\frac{\partial}{\partial v} \phi(v,x_1,\ldots,x_i) - x_i \phi(v,x_1,\ldots,x_{i}) = \phi(v,x_1,\ldots,x_{i-1}).
\]
\end{lemma}
\begin{proof}
For $n\geq i$, $x_i h_{n-i}(x_1,\ldots,x_i)$ is the sum over monomials of degree $n-i+1$ in the variables $x_1,\ldots,x_i$ such that $x_i$ appears with degree at least 1. However, $h_{n-i+1}(x_1,\ldots,x_i) - h_{n-i+1}(x_1,\ldots,x_{i-1})$ is the sum over the same set of monomials. When $n=i-1$, both $h_{n-i+1}(x_1,\ldots,x_i)$ and $h_{n-i+1}(x_1,\ldots,x_{i-1})$ are equal to 1, and $x_ih_{n-i}(x_1,\ldots,x_i)=0$. Multiplying the equation
\[
x_i h_{n-i}(x_1,\ldots,x_i) = h_{n-i+1}(x_1,\ldots,x_i) - h_{n-i+1}(x_1,\ldots,x_{i-1})
\]
by $\frac{v^n}{n!}$ and summing over $n\geq i-1$ yields the desired result. 
\end{proof}
 
Two specializations of $\phi$ are used in the analysis of the differential operators corresponding to $(a,b,c,d)$-dipoles and $(a,b,0,0)$-dipoles. 
\begin{defn}[$\phi_S(v)$ and $\phi_{i,j}(v)$]
Let $S \in \{\bdot,\wdot\}^i$, say, $S = S_1S_2\cdots S_i$. Define the function $\phi_S(v)$ to be the evaluation of $\phi(v,x_1,\ldots,x_i)$, given by Definition \ref{defn:general-phi}, at
 \[
 x_j = \begin{cases}
 f_{(\wdot)} & \text{~if } S_j=\wdot, \\
 w f_{(\bdot)} & \text{~if } S_j = \bdot.
 \end{cases}
 \]
 Let $\phi_{i,j}(v) = \phi_S(v)$ whenever $S$ is a string consisting of $i$ white dots and $j$ black dots.
 \end{defn}
The notation $\phi_{i,j}(v)$ is well-defined since $\phi$ is symmetric in $x_1,\ldots,x_i$. (Nevertheless, in many cases it will often be more natural to index $\phi$ by $S$ instead of $i$ and $j$; thus, both notations are used.) The specialization of Lemma \ref{lem:fundamental-phi-property} to $\phi_S(v)$ is as follows. 
 \begin{cor}
 \label{cor:phi-first-specialization}
 Let $i\geq 1$ and let $S \in \{\wdot,\bdot\}^{i-1}$. Then
 \[
 \frac{\partial}{\partial v} \phi_{S\wdot}(v) - f_{(\wdot)} \phi_{S\wdot}(v) = \phi_S(v)
\quad\text{ and }\quad
 \frac{\partial}{\partial v} \phi_{S\bdot}(v) - f_{(\bdot)} w\phi_{S\bdot}(v) = \phi_S(v).
 \]
 \end{cor}
A further specialization is the following.
\begin{defn}[$\phi_i(y)$]
Let $\phi_i(y)$ denote the evaluation of $\phi(v,x_1,\ldots,x_i)$, given by Definition \ref{defn:general-phi}, at $v=y$ and $x_j= f_{(\bdot)}$ for $1\leq j\leq i$. 
\end{defn}
In this case, Lemma \ref{lem:fundamental-phi-property} specializes as follows.
 \begin{cor}
 \label{cor:phi-second-specialization}
 Let $i\geq 0$. Then 
 \[
 \frac{\partial \phi_{i+1}}{\partial y} - f_{(\bdot)} \phi_{i+1} = \phi_i.
 \]
 \end{cor}
Since our strategy involves applying $J'$ and $J''$ to series involving $\phi_i$ and $\phi_S$, it is necessary to determine the derivatives of these functions with respect to $f_{\bdot}$ and $f_{\wdot}$. This is given in the following.
\begin{lemma}
Let $i,j\geq 1$. Then
\[
 \frac{\partial \phi_{i,j}}{\partial f_{(\wdot)}} = i \phi_{i+1,j}, \quad
 \frac{\partial \phi_{i,j}}{\partial f_{(\bdot)}} = jw \phi_{i,j+1}, \quad\text{and }\quad
\frac{\partial \phi_i}{\partial f_{(\bdot)}} = i \phi_{i+1}
\]
\label{lem:phi-derivative}
\end{lemma}
\begin{proof}
We give a proof of the first equation, noting that the other two may be proven in a similar manner. By definition,
 \begin{align*}
\left[ \frac{v^n}{n!} \right] \phi_{i,j}  &= [t^{n-i-j}] (1-tf_{(\wdot)})^{-i} (1-twf_{(\bdot)})^{-j}.
 \end{align*}
If $i=0$, then $\frac{\partial \phi_{0,j}}{\partial f_{(\wdot)}}=0$. For $i>1$, and any $n\geq i$, 
 \begin{align*}
 \left[ \frac{v^n}{n!} \right]  \frac{\partial \phi_{i,j}}{\partial f_{(\wdot)}} &= [t^{n-i}] \frac{\partial}{\partial f_{(\wdot)}}  (1-tf_{(\wdot)})^{-i} (1-tyf_{(\bdot)})^{-j} 
 = i \left[\frac{v^n}{n!}\right]\phi_{i+1,j}.
 \end{align*}
 Thus,
 \[
 \frac{\partial \phi_{i,j}}{\partial f_{(\wdot)}} = i \phi_{i+1,j}
 \]
 when $i>0$. 
\end{proof}

\subsection{Preimage series for the $(a,b,c,d)$-dipole problem}
 
We now return to the question of determing the preimages of $g_R \phi_i$ and $g_R \phi_S$ under $\partial/\partial y - C'$ and $\partial/\partial v - C''$, respectively.

\subsubsection{The preimage under $\partial/\partial y - C'$}
\label{sec:ab00-preimages}

The following series play an important role in determining the solution to the $(a,b,0,0)$-dipole problem. They form a natural set of functions in which the $(a,b,0,0)$-dipole series may be expressed.
\begin{defn}[Preimage series for $\partial/\partial y - C'$]
Let $\tau_{R,k}(y)$ be the unique solution to the partial differential equation
\begin{equation}
\label{eqn:tauRk-definition}
\left(\frac{\partial}{\partial y} - C'\right) \tau_{R,k} = g_{\bdot R} \phi_k(y).
\end{equation}
\end{defn}
These functions may be expressed as a sum indexed by the following combinatorial objects.
\begin{defn}
Let $R \in \{\bdot,\wdot\}^*$ be a binary string. An ordered sequence $\strc = (\strc_1,\strc_2,\ldots,\strc_i)$ such that each $\strc_i\in \{\bdot,\wdot\}^*\setminus \epsilon $ and $\strc_1\strc_2\cdots\strc_i = R$ is called a string composition of $R$ into $k$ parts. Let $\mathcal{C}_i(R)$ denote the set of string compositions of $R$ into $i$ parts. For a part $\strc_k$ of a string composition, $\strc_{k,1}$ shall denote the first symbol of $\strc_k$.
\end{defn}
The empty string in  $\{\bdot,\wdot\}^*$ is denoted by $\epsilon$. We then have the following expression for $\tau_{R,k}$.
\begin{thm}
\label{thm:tau-nonrecursive}
Let $R\in \{\bdot,\wdot\}^*\setminus \epsilon$, and $k\geq 0$. Then
\begin{equation}
\label{eqn:tau-nonrecursive}
\tau_{R,k} = \sum_{i\geq 1} \phi_{k+i} \sum_{(\strc_1,\ldots,\strc_i) \in\mathcal{C}_i(R)} g_{\bdot \strc_1} \prod_{2\leq j\leq i} f_{(\bdot \strc_j)}.
\end{equation}
\end{thm}
\begin{proof}
The strategy for determining $\tau_{R,k}$ described in Section \ref{sec:strategy} leads to the following recursion for $\tau_{R,k}$:
\begin{equation}
\label{eqn:tau-recursion}
\tau_{R,k} = g_{\bdot R} \phi_{k+1} + \sum_{1\leq i \leq \ell(R)-1} f_{(\bdot R_{i+1}\cdots R_{\ell(R)})}\tau_{ R_1\cdots R_i,k+1}.
\end{equation}
The expression given in (\ref{eqn:tau-nonrecursive}) is the closed form of this recursion. This will be proven by demonstrating that applying $\frac{\partial}{\partial y}- C'$ to the right hand side of (\ref{eqn:tau-nonrecursive}) yields $g_{\bdot R} \phi_k(y)$. 

For a string composition $\strc\in \mathcal{C}_i(R)$, let
\[
W_i(\strc):= \prod_{2\leq j\leq i} f_{(\bdot\strc_j)}.
\]
Applying $C'$ to the right side of Equation (\ref{eqn:tau-nonrecursive}) gives
\begin{align*}
& \sum_{i\geq 1} \phi_{k+i} \sum_{(\strc_1,\ldots,\strc_i) \in\mathcal{C}_i(R)} \sum_{1\leq n\leq \ell(\strc_1)} g_{\bdot \strc_{1,1}\cdots \strc_{1,n}} f_{(\bdot \strc_{1,n+1}\cdots \strc_{1,\ell(\strc)})} W_i(\strc) \\
&\quad = f_{(\bdot)} \sum_{i\geq 1} \phi_{k+i}  \sum_{(\strc_1,\ldots,\strc_i) \in\mathcal{C}_i(R)} g_{\bdot \strc_1} W_i(\strc) \\
&\qquad +  \sum_{i\geq 1} \phi_{k+i} \sum_{(\strc_1,\ldots,\strc_i) \in\mathcal{C}_i(R)} \sum_{1\leq n\leq \ell(\strc_1)-1} g_{\bdot \strc_{1,1}\cdots \strc_{1,n}} f_{(\bdot \strc_{1,n+1}\cdots \strc_{1,\ell(\strc)})} W_i(\strc) \\
&\quad = f_{(\bdot)} \sum_{i\geq 1} \phi_{k+i}  \sum_{(\strc_1,\ldots,\strc_i) \in\mathcal{C}_i(R)} g_{\bdot \strc_1} W_i(\strc) \\
&\qquad + \sum_{i\geq 1} \phi_{k+i} \sum_{(\strc_1,\ldots,\strc_{i+1})\in \mathcal{C}_{i+1}(R)} g_{\bdot \strc_1} W_{i+1}(\strc),
\end{align*}
by relabelling $\strc_{1,1}\cdots \strc_{1,n}$ to be $\strc_1$, and $\strc_{1,n+1}\cdots \strc_{1,\ell(\strc)}$ to be $\strc_{i+1}$. Thus,
\begin{align*}
&\left( \frac{\partial}{\partial y} - C' \right) \sum_{i\geq 1} \phi_{k+i} \sum_{(\strc_1,\ldots,\strc_i) \in\mathcal{C}_i(R)} g_{\bdot \strc_1} W_i(\strc) \\
&\quad= \sum_{i\geq 1} \left(\frac{\partial \phi_{k+i}}{\partial y} - f_{(\bdot)} \phi_{k+i} \right)\sum_{(\strc_1,\ldots,\strc_i) \in\mathcal{C}_i(R)} g_{\bdot \strc_1} W_i(\strc) \\
&\qquad -  \sum_{i\geq 1} \phi_{k+i} \sum_{(\strc_1,\ldots,\strc_{i+1})\in \mathcal{C}_{i+1}(R)} g_{\bdot \strc_1} W_{i+1}(\strc) \\
&\quad = \sum_{i\geq 1} \phi_{k+i-1 } \sum_{(\strc_1,\ldots,\strc_i) \in\mathcal{C}_i(R)} g_{\bdot \strc_1} W_i(\strc)  \\
&\qquad - \sum_{i\geq 2} \phi_{k+i-1} \sum_{(\strc_1,\ldots,\strc_{i})\in \mathcal{C}_{i}(R)} g_{\bdot \strc_1} W_i(\strc)\\
&\quad = \phi_k g_{\bdot R}.
\end{align*}
Since the right side of Equation (\ref{eqn:tau-nonrecursive}) vanishes when $x=0$, then by uniqueness of the solution to the partial differential equation, it must be equal to $\tau_{R,k}$. 
\end{proof}

\subsubsection{The preimage under $\partial/\partial v - C''$}
\label{sec:abcd-preimages}

We now describe the preimages of $g_R \phi_S$ under the operator $\partial/\partial v - C''$, starting by introducing the following notation.
\begin{defn}[Preimage series for $\partial/\partial v - C''$]
Define the functions $\tau_{R,S}$ to be the unique solution to the partial differential equation
\[
\left( \frac{\partial}{\partial v} - C'' \right) \tau_{R,S} = g_{\bdot R} \phi_S(v)
\]
which satisfies the initial condition $\tau_{R,S}|_{v=0}=0$. 
\end{defn}
Notationally, these are distinguished from the $\tau$-functions in the preceding section by the fact that they are indexed by a pair of binary strings, as opposed to a binary string and an integer. For a string composition $(\strc_1,\ldots,\strc_k)$ of $R$, let $\strc_{j,1}$ denote the first symbol of $\strc_j$, and define
\[
\iota(\strc_1,\ldots,\strc_k):= |\{2\leq j\leq k : \strc_{j,1} = \bdot\}|.
\]
Then the functions $\tau_{R,S}$ may be expressed as follows.
\begin{thm}
\label{thm:sigma-nonrecursive}
Let $R\in \{\bdot,\wdot\}^*\setminus \epsilon$. Let $S\in \{\bdot,\wdot\}^*$. Then
\begin{equation}
\label{eqn:sigma-nonrecursive}
\tau_{R,S} = \sum_{k\geq 1} \left( \localvari_{1,k} + \localvari_{2,k}\right),
\end{equation}
where 
\begin{align*}
\localvari_{1,k} &= \sum_{(\strc_1,\ldots,\strc_k) \in \mathcal{C}_k(R)} w^{\iota(\strc_1,\ldots,\strc_k)}g_{\bdot \strc_k}\Omega(\strc) \phi_{\strc_{k,1}\cdots \strc_{1,1}S},\\
\localvari_{2,k} &= \sum_{(\strc_1,\ldots,\strc_k) \in \mathcal{C}_k(R)} w^{\iota(\strc_1,\ldots,\strc_k)+1}g_{\bdot} f_{(\bdot \strc_k)} \Omega(\strc) \phi_{\bdot \strc_{k,1}\cdots \strc_{1,1}S} ,
\end{align*}
and
\[
\Omega(\strc) := \prod_{1\leq j\leq k-1} f_{(\strc_j \strc_{j+1,1})}.
\]
Furthermore, $\tau_{\epsilon,S} = g_{\bdot} \phi_{\bdot S}$. 
\end{thm}
\begin{proof}
The verification that $\tau_{\epsilon,S} = g_{\bdot} \phi_{\bdot S}$ is a routine calculation using Corollary~\ref{cor:phi-first-specialization}.
As in the case for $\tau_{R,k}$, the functions $\tau_{R,S}$ satisfy the recursion
\begin{equation}
\label{eqn:sigma-recursion}
\tau_{tR,S}(v) = \phi_{tS}(v) g_{ \bdot tR} + \sum_{1\leq i \leq \ell(R)} f_{(tR_1\cdots R_i)} w^{\delta_{R_i,\bdot}}\tau_{R_i\cdots R_{\ell(R)},tS} + w f_{(\bdot tR)} \tau_{\epsilon, tS},
\end{equation}
the closed form of which is given by (\ref{eqn:sigma-nonrecursive}). We prove this by applying $\frac{\partial}{\partial v} - C''$ to the right side of (\ref{eqn:sigma-nonrecursive}). 

Applying $C''$ to the first of these expressions gives
\begin{align*}
C'' \localvari_{1,k} 
&= \localvarii_{1,k} + \localvarii_{2,k} + \localvarii_{3,k},
\end{align*}
where
\begin{align*}
\localvarii_{1,k} &= \sum_{(\strc_1,\ldots,\strc_k) \in \mathcal{C}_k(R)} w^{\iota(\strc_1,\ldots,\strc_k) +1} g_{\bdot} f_{(\bdot \strc_k)} \Omega(\strc) \phi_{\strc_{k,1}\cdots \strc_{1,1}S},\\
\localvarii_{2,k} &=  \sum_{(\strc_1,\ldots,\strc_k) \in \mathcal{C}_k(R)} w^{\iota(\strc_1,\ldots,\strc_k)  + \delta_{\strc_{k,1},\bdot} }g_{\bdot \strc_k} f_{(\strc_{k,1})} \Omega(\strc) \phi_{\strc_{k,1}\cdots \strc_{1,1}S},
\end{align*}
and
\[
\localvarii_{3,k} =  \sum_{(\strc_1,\ldots,\strc_k) \in \mathcal{C}_k(R)} \sum_{2\leq i\leq \ell(\strc_k)} w^{\iota(\strc_1,\ldots,\strc_k)  + \delta_{\strc_{k,i},\bdot}} g_{\bdot \strc_{k,i}\cdots \strc_{k,\ell(\strc_k)}} f_{(\strc_{k,1}\cdots\strc_{k,i})} \Omega(\strc) \phi_{\strc_{k,1}\cdots \strc_{1,1}S} .
\]
Applying Corollary \ref{cor:phi-first-specialization},
\begin{align*}
\frac{\partial}{\partial v} \localvari_{1,k} - \localvarii_{2,k} &= \sum_{(\strc_1,\ldots,\strc_k) \in \mathcal{C}_k(R)} w^{\iota(\strc_1,\ldots,\strc_k)} g_{\bdot \strc_k} \Omega(\strc)  \\
&\qquad\qquad\qquad\qquad\qquad \times \left(\frac{\partial}{\partial v} \phi_{\strc_{k,1}\cdots \strc_{1,1}S} - w^{\delta_{\strc_{k,1}},\bdot} f_{(\strc_{k,1})} \phi_{\strc_{k,1}\cdots \strc_{1,1}S} \right)  \\
&=  \sum_{(\strc_1,\ldots,\strc_k) \in \mathcal{C}_k(R)} w^{\iota(\strc_1,\ldots,\strc_k)} g_{\bdot \strc_k} \Omega(\strc) \phi_{\strc_{k-1,1}\cdots \strc_{1,1} S}.
\end{align*}
To analyze $\localvarii_{k,3}$, given any $(\strc_1,\ldots,\strc_k)\in \mathcal{C}_k(R)$ and $2\leq i\leq \ell(\strc_k)$, define the string composition $(\strc'_1,\ldots,\strc'_k,\strc'_{k+1}) \in \mathcal{C}_{k+1}(R)$ by $\strc'_j = \strc_j$ when $1\leq j\leq k-1$, $\strc'_k = \strc_{k,1}\cdots \strc_{k,i-1}$ and $\strc'_{k+1} = \strc_{k,i}\cdots \strc_{k,\ell(\strc_k)}$. Since $c(\strc'_1,\ldots,\strc'_{k+1}) = c(\strc_1,\ldots,\strc_k)  + \delta_{\strc_{k,i},\bdot}$, then
\[
\localvarii_{3,k} = \sum_{(\strc'_1,\ldots,\strc'_{k+1})\in \mathcal{C}_{k+1}(R)} w^{c(\strc'_1,\ldots,\strc'_{k+1})} g_{\bdot \strc'_{k+1}} \phi_{\strc'_{k,1}\cdots \strc'_{1,1}S} \prod_{1\leq j\leq k} f_{(\strc'_j \strc'_{j+1,1})}.
\]
In other words, for $k\geq 2$,
\[
\frac{\partial}{\partial v} \localvari_{1,k} - \localvarii_{2,k} - \localvarii_{3,k-1} = 0.
\]

Next, apply $\left(\frac{\partial}{\partial v} -C'' \right)$ to $\localvari_{2,k}$ and use Corollary \ref{cor:phi-first-specialization}  to obtain
\begin{align*}
\left(\frac{\partial}{\partial v} -C''\right) \localvari_{2,k}  &=  \sum_{(\strc_1,\ldots,\strc_k) \in \mathcal{C}_k(R)} w^{c(\strc_1,\ldots,\strc_k)+1}g_{\bdot} f_{(\bdot \strc_k)} \Omega(\strc)\\
&\qquad\qquad\qquad\qquad \times \left( \frac{\partial}{\partial v} \phi_{\bdot \strc_{k,1}\cdots \strc_{1,1}S} - wf_{(\bdot)} \phi_{\bdot \strc_{k,1}\cdots \strc_{1,1}S}\right) \\
&=  \sum_{(\strc_1,\ldots,\strc_k) \in \mathcal{C}_k(R)} w^{c(\strc_1,\ldots,\strc_k)+1}g_{\bdot} f_{(\bdot \strc_k)} \Omega(\strc) \phi_{\strc_{k,1}\cdots \strc_{1,1}S}  \\
&= \localvarii_{1,k}.
\end{align*}

Consequently, applying $\left( \frac{\partial}{\partial v} - C'' \right)$ to the right side of Equation (\ref{eqn:sigma-nonrecursive}) yields
\begin{align*}
\left( \frac{\partial}{\partial v} - C'' \right) \sum_{k\geq 1} \localvari_{1,k} + \localvari_{2,k} &= \sum_{k\geq 1}\left( \frac{\partial}{\partial v} \localvari_{1,k}  - \localvarii_{1,k} - \localvarii_{2,k} - \localvarii_{3,k}+ \left(\frac{\partial}{\partial v} - C'' \right) \localvari_{2,k}\right) \\
&=  \frac{\partial}{\partial v}\localvari_{1,1} - \localvarii_{2,1} \\
&= g_{\bdot R}\frac{\partial}{\partial v} \phi_{R_1S}  - w^{\delta_{R_1,\bdot}} g_{\bdot R} f_{(R_1)} \phi_{R_1S} \\
&= g_{\bdot R} \phi_S,
\end{align*}
again using Corollary \ref{cor:phi-first-specialization}. Since the right side of Equation (\ref{eqn:sigma-nonrecursive}) vanishes when $v=0$, it is the unique solution to $\left( \frac{\partial}{\partial v} - C'' \right) \tau_{R,S} = g_{\bdot R} \phi_S.$
\end{proof}

As a demonstration of Theorem \ref{thm:sigma-nonrecursive}, Table \ref{table:tau-contributions} gives the contribution to $\tau_{R,S}$ from various string compositions when $\ell(R)\leq 2$. The function $\tau_{R,S}$ is the sum of the terms in the third column which correspond to it.

\begin{table}
\[
\begin{array}{c|c|c}
\tau_{R,S} & \text{String Composition} & \text{Contribution to } \tau_{R,S} \\
\hline
\tau_{\bdot,S} & (\bdot) & g_{\bdot\bdot} \phi_{\bdot S} + wg_{\bdot} f_{(\bdot\bdot)} \phi_{\bdot\bdot S} \\
\hline
\tau_{\wdot,S} & (\wdot) & g_{\bdot\wdot} \phi_{\wdot S} + wg_{\bdot} f_{(\bdot \wdot)} \phi_{\bdot \wdot S} \\
\hline
\tau_{\bdot\bdot,S} & (\bdot\bdot) & g_{\bdot\bdot\bdot} \phi_{\bdot S} + wg_{\bdot} f_{(\bdot\bdot\bdot)} \phi_{\bdot\bdot S} \\ 
& (\bdot,\bdot) & w (g_{\bdot\bdot} \phi_{\bdot\bdot S} + w g_{\bdot} f_{(\bdot\bdot)} \phi_{\bdot\bdot \bdot S} ) f_{(\bdot\bdot)} \\
\hline
\tau_{\wdot\bdot,S} & (\wdot \bdot) & g_{\bdot\wdot\bdot}\phi_{\wdot S} + wg_{\bdot} f_{(\bdot\wdot\bdot)}\phi_{\bdot\wdot S} \\
& (\wdot, \bdot) & w(g_{\bdot\bdot} \phi_{\bdot \wdot S} + wg_{\bdot} f_{(\bdot \bdot)} \phi_{\bdot \bdot\wdot S}) f_{(\wdot \bdot)} \\
\hline
\tau_{\bdot\wdot,S} & (\bdot\wdot) &  g_{\bdot\bdot\wdot} \phi_{\bdot S} + wg_{\bdot} f_{(\bdot\bdot\wdot)}\phi_{\bdot\bdot S} \\
& (\bdot, \wdot) & (g_{\bdot\wdot} \phi_{\wdot\bdot S} + w g_{\bdot} f_{(\bdot \wdot)} \phi_{\bdot\wdot\bdot S}) f_{(\bdot \wdot)} \\
\hline
\tau_{\wdot\wdot,S} & (\wdot\wdot) & g_{\bdot\wdot\wdot} \phi_{\wdot S} + w g_{\bdot} f_{(\bdot \wdot\wdot)}\phi_{\bdot\wdot S} \\
& (\wdot,\wdot) & (g_{\bdot \wdot}\phi_{\wdot\wdot S} + w g_{\bdot} f_{(\bdot\wdot)} \phi_{\bdot\wdot\wdot S}) f_{(\wdot \wdot)}
\end{array}
\]
\caption{Contributions to $\tau_{R,S}$ from various small string compositions.} 
\label{table:tau-contributions}
\end{table}

\section{Solutions for Low Genera}
\label{sec:small-genus-solutions}

In this section, we use the strategy described in Section \ref{sec:strategy} to give explicit expressions for the $(a,b,c,d)$-dipole series $\langle \serC^{(g)}\rangle$, when $g=1$ and $g=2$, as a linear combination of the functions $\langle \tau_{R,S}\rangle$. We give $\langle \serC^{(1)}\rangle$ in Corollary \ref{cor:abcd-torus-specialized}, and we discuss the derivation of $\langle \serC^{(2)}\rangle$ in Section \ref{subsec:double-torus-solution}, with the final result appearing in Appendix \ref{sec:appendix-A}. Extracting coefficients from these series to obtain expressions for the number of $(a,b,c,d)$ dipoles on the given surface. These coefficients may be used to determine the number of $(p,q,n)$-dipoles on the surface using the formula
\[
[r^ps^qt^nu^{2g}] \Phi = \sum_{0\leq b \leq p-1}  \left[\frac{v^{n-q-1}}{(n-q-1)!}\frac{y^b}{b!} w^{p-1-b}\frac{x^{q-b}}{(q-b)!}u^{2g}\right] \langle \serC \rangle,
\]
given by Proposition \ref{prop:abcd-to-pqn}.

\subsection{Solutions for the Torus}
\label{subsec:torus-solution}

In order to determine $\serC^{(1)}$, we first require the generating series for $(a,b,0,0)$-dipoles on the torus.
\begin{lemma}
The generating series for $(a,b,0,0)$-dipoles on the torus is 
\begin{align*}
\serB^{(1)} &= \tau_{\bdot\wdot,0} x\exp(x f_{(\wdot)}) + g_{\bdot} \left( \sum_{n\geq 3} \frac{x^n}{n!}  \psi_{(3,1^{n-3})} f_{(\wdot\wdot\wdot)}f_{(\wdot)}^{n-3} + \sum_{n\geq 4}  \psi_{(2^2,1^{n-4})} f_{(\wdot\wdot)}^2 f_{(\wdot)}^{n-4}\right),\end{align*}
where $\psi_{\lambda}$ is given by Equation (\ref{eqn:psi-expression}).
\label{lemma:ab00-on-torus}
\end{lemma}
\begin{proof}
When solving for the genus 1 series, Equation (\ref{eqn:ab00-diff-recursion-2}) becomes
\begin{align*}
\left(\frac{\partial}{\partial y} - C'\right) \serB^{(1)} &= J g_{\bdot}(\exp(xf_{(\wdot)})-1) = g_{\bdot\bdot \wdot} x \exp(x f_{(\wdot)}),
\end{align*} 
with initial condition $\hat{\serB}^{(1)}|_{y=0}=0$. By Theorem~\ref{thm:tau-nonrecursive}, the solution to this equation is 
\[
\hat{\serB}^{(1)} = \tau_{\bdot\wdot,0} x\exp(x f_{(\wdot)}),
\]
where
\[
\tau_{\bdot\wdot,0} = \phi_1(y) g_{\bdot\bdot\wdot} + \phi_2(y) g_{\bdot\bdot} f_{(\bdot \wdot)}.
\]
The series $\serB^{(1)}|_{y=0}$ can be obtained from the initial condition given in Lemma~\ref{lemma:ab00-PDE}. Using the fact that genus 1 dipoles with $n$ edges must have half face-degree sequence either $(3,1^{n-3})$ or $(2^2,1^{n-4})$,
\begin{align*}
\serB^{(1)}|_{y=0} &= [u^2] g_{\bdot} \sum_{D \in \mathcal{D}} \frac{x^{n(D)}}{n(D)!} u^{2g(D)} f_{\lambda'(D)} \\
&= g_{\bdot} \left( \sum_{n\geq 3} \frac{x^n}{n!}  \psi_{(3,1^{n-3})} f_{(\wdot\wdot\wdot)}f_{(\wdot)}^{n-3} + \sum_{n\geq 4}  \psi_{(2^2,1^{n-4})} f_{(\wdot\wdot)}^2 f_{(\wdot)}^{n-4}\right).
\end{align*}
 Since $\serB^{(1)} = \hat{\serB}^{(1)} + \serB^{(1)}|_{y=0}$, the result has now been proven.
\end{proof}
Although $\psi_{\lambda}$ is can be determined using Equation (\ref{eqn:psi-expression}), it simplifies computations if it is left unevaluated while solving the differential equations for $(a,b,c,d)$-dipoles. This has the additional benefit  that the results obtained will appear as ``linear combinations of central problems.'' Furthermore, the $\psi_{\lambda}$'s allow us to identify which parts of the solutions arise from the central initial condition and which parts arise from the non-central aspects of the problem.

The generating series for $(a,b,0,0)$-dipoles in the torus may be used to determine the genus 1 solution to the $(a,b,c,d)$-dipole problem, as follows.
\begin{thm}
The generating series for $(a,b,c,d)$-dipoles in the torus is
\begin{align*}
\serC^{(1)} &= (wf_{(\bdot)} \tau_{\wdot\wdot, \bdot} + \tau_{\wdot\wdot,\epsilon}) x\exp(xf_{(\wdot)}) \\
&\quad+ (w^2 \tau_{\bdot\bdot,\bdot} + w^3 f_{(\bdot)} \tau_{\bdot\bdot,\bdot\bdot}) (\exp(xf_{(\wdot)})-1) \\
&\quad+ (wf_{(\bdot\bdot\wdot)} \tau_{\epsilon,\epsilon} + w f_{(\bdot)} \tau_{\bdot\wdot,\epsilon} + f_{(\bdot\wdot)} \tau_{\wdot,\epsilon}) \phi_1(y) x\exp(xf_{(\wdot)}) \\
&\quad+ (wf_{(\bdot\bdot)} \tau_{\epsilon,\epsilon}+ w f_{(\bdot)} \tau_{\bdot,\epsilon} ) f_{(\bdot\wdot)}\phi_2(y)x\exp(xf_{(\wdot)}) \\
&\quad+ wf_{(\bdot)} \tau_{\epsilon,\epsilon} \left( \sum_{n\geq 3} \frac{x^n}{n!}  \psi_{(3,1^{n-3})} f_{(\wdot\wdot\wdot)}f_{(\wdot)}^{n-3} + \sum_{n\geq 4}  \frac{x^n}{n!} \psi_{(2^2,1^{n-4})} f_{(\wdot\wdot)}^2 f_{(\wdot)}^{n-4}\right) + \serB^{(1)},
\end{align*}
where the functions $\tau_{R,S}$ are given by Theorem \ref{thm:sigma-nonrecursive}, and $\serB^{(1)}$ is given by Lemma \ref{lemma:ab00-on-torus}.
\label{thm:abcd-torus}
\end{thm}

\begin{proof}
From Equation (\ref{eqn:psi-doubleprime-no-ic}), to determine the generating series for $(a,b,c,d)$-dipoles on the torus, first, determine $J'' \serC^{(0)} + C''\serB^{(1)}$. Using Lemma \ref{lem:phi-derivative},
\begin{align*}
J'' \serC^{(0)}&= g_{\bdot\wdot\wdot}(wf_{(\bdot)}\phi_{\bdot}(v) + 1) x\exp(xf_{(\wdot)}) \\
&\quad+g_{\bdot\bdot\bdot}(w^2 \phi_{\bdot}(v) + w^3 f_{(\bdot)} \phi_{\bdot\bdot}(v)) (\exp(x f_{(\wdot)})-1).
\end{align*}
Using the expression for $\serB^{(1)}$ given in Lemma \ref{lemma:ab00-on-torus},
\begin{align*}
C''\serB^{(1)} &= (w g_{\bdot} f_{(\bdot\bdot\wdot)} + wg_{\bdot\bdot\wdot} f_{(\bdot)} + g_{\bdot\wdot} f_{(\bdot\wdot)}) \phi_1(y)x\exp(xf_{(\wdot)}) \\
&\quad+ (wg_{\bdot} f_{(\bdot\bdot)} + w g_{\bdot\bdot} f_{(\bdot)}) f_{(\bdot\wdot)} \phi_2(y)x\exp(xf_{(\wdot)}) \\
&\quad+ wg_{\bdot} f_{(\bdot)} \left( \sum_{n\geq 3} \frac{x^n}{n!}  \psi_{(3,1^{n-3})} f_{(\wdot\wdot\wdot)}f_{(\wdot)}^{n-3} + \sum_{n\geq 4} \frac{x^n}{n!} \psi_{(2^2,1^{n-4})} f_{(\wdot\wdot)}^2 f_{(\wdot)}^{n-4}\right).
\end{align*}
Expressing the solution to Equation (\ref{eqn:psi-doubleprime-no-ic}) in terms of the functions $\tau_{R,S}$ and adding the initial condition gives the result.
\end{proof}

When information about face structure is forgotten, the following series is obtained.

\begin{cor}
\begin{align*}
\langle \serC^{(1)}\rangle &= xe^x(\langle \phi_{1,0 }\rangle + \langle \phi_{2,0 }\rangle + 2w \langle \phi_{1,1 }\rangle + 2w\langle \phi_{2,1 }\rangle + w^2\langle \phi_{1,2}\rangle + w^2\langle \phi_{2,2 }\rangle )\\
&\quad+ (e^x-1) ( w^2 \langle \phi_{0,2 }\rangle + 3w^3\langle \phi_{0,3}\rangle + 3w^4 \langle \phi_{0,4}\rangle + w^5 \langle \phi_{0,5}\rangle) \\
&\quad+ xe^x \langle \phi_1(y)\rangle (\langle \phi_{1,0}\rangle + 2w \langle \phi_{0,1 }\rangle + 2w\langle \phi_{1,1}\rangle + w^2\langle \phi_{0,2}\rangle + w^2 \langle \phi_{1,2}\rangle) \\
&\quad+ xe^x \langle \phi_2(y)\rangle (2w\langle \phi_{0,1}\rangle + w^2 \langle \phi_{0,2}\rangle) \\
&\quad+ w\langle \phi_{0,1}\rangle \left( \sum_{n\geq 3} \frac{x^n}{n!}  \psi_{(3,1^{n-3})}  + \sum_{n\geq 4}  \frac{x^n}{n!} \psi_{(2^2,1^{n-4})} \right) + \langle \serB^{(1)}\rangle.
\end{align*}
\label{cor:abcd-torus-specialized}
\end{cor}

\subsubsection{Extracting Coefficients from the $(a,b,0,0)$-dipole series}

In order to obtain the generating series with respect to the weights $a, b, c, d, n,$ and $g$ (but not with respect to face-degree sequence), we set all $f$- and $g$-type indeterminates to 1. Recall that, given a formal power series $F$, we use $\langle F \rangle$ to denote the series so obtained.

In order to extract coefficients from the series appearing in Lemma \ref{lemma:ab00-on-torus}, the first step is to determine the coefficients of $\langle \tau_{S,k}\rangle$. The functions $\langle \tau_{S,k} \rangle$ are expressed in terms of $\langle \phi_i \rangle$, whose coefficients are, from definition, given by the following.
\begin{lemma}
 \[
 \left[ \frac{y^n}{n!}\right] \langle \phi_i \rangle = [t^{n-i}] (1-t )^{-i} = \begin{cases}
 \binom{n-1}{n-i}& \text{ if } n\geq i, \\
 0 & \text{ if } n<i.
 \end{cases}
 \]
 \label{lemma:phi-coefficient-1}
\end{lemma}
When setting the $f$- and $g$-type indeterminates to 1, $\tau_{R,k}$ has the following more explicit form.
\begin{cor}
\label{cor:tau-coefficients}
Let $R \in \{\bdot,\wdot\}^*\setminus \epsilon$, and $k\geq 0$. Then
\[
\langle \tau_{R,k} \rangle = \sum_{i\geq 1} \binom{\ell(R)-1}{i-1} \langle \phi_{k+i}(x) \rangle 
\text{ and }
\left[ \frac{y^n}{n!}\right] \langle \tau_{R,k}\rangle = \binom{\ell(R)+b-2}{\ell(R)+k-1}.
\]
\end{cor}
\begin{proof}
The first expression follows from Theorem \ref{thm:tau-nonrecursive}, along with the observation that the number of string compositions of $R$ into $i$ parts is equal to the number of integer compositions of $\ell(R)$ into $i$ parts, which is $\binom{\ell(R)-1}{i-1}$. Extracting the coefficient of $\frac{y^n}{n!}$, 
\[
\left[ \frac{y^n}{n!}\right] \langle \tau_{R,k}\rangle = \sum_{i\geq 1} \binom{\ell(R)-1}{i-1} \binom{n-1}{n-k-i} 
= \binom{\ell(R)+n-2}{\ell(R)+k-1},
\]
by Vandermonde's identity. 
\end{proof}
This result may be used with Lemma \ref{lemma:ab00-on-torus}, to determine the number of $(a,b,0,0)$-dipoles on the torus when $b\geq 1$:
\begin{equation}
\label{eqn:ab00-torus-coeff}
\left[\frac{y^b}{b!} \frac{x^{a+1}}{(a+1)!}\right] \langle \serB^{(1)} \rangle = \langle \tau_{\bdot\wdot,0}\rangle x\exp(x) = b(a+1).
\end{equation}

\subsubsection{Extracting Coefficients from the $(a,b,c,d)$-dipole series}

The number of $(a,b,c,d)$-dipoles on the torus may be obtained by extracting coefficients from the expression for $\langle\serC^{(1)}\rangle$ given in Corollary \ref{cor:abcd-torus-specialized}. Suppose $c+d>0$. Then the term $\langle \serB^{(1)}\rangle$ in $\langle\serC^{(1)}\rangle$ may be disregarded. 
Extracting the coefficients of $x^{a+1}$ and $y^b$ may be done as in the $(a,b,0,0)$ case:
\begin{align*}
\left[ \frac{x^{a+1}y^b}{(a+1)!b!}\right] \langle \serC^{(1)} \rangle &= (a+1)\delta_{b,0}(\langle \phi_{1,0 }\rangle + \langle \phi_{2,0 }\rangle + 2w \langle \phi_{1,1 }\rangle + 2w\langle \phi_{2,1 }\rangle + w^2\langle \phi_{1,2}\rangle + w^2\langle \phi_{2,2 }\rangle )\\
&\quad+ \delta_{b,0}( w^2 \langle \phi_{0,2 }\rangle +  3w^3\langle \phi_{0,3}\rangle + 3w^4 \langle \phi_{0,4}\rangle + w^5 \langle \phi_{0,5}\rangle) \\
&\quad+ (a+1)(1-\delta_{b,0}) (\langle \phi_{1,0}\rangle + 2w \langle \phi_{0,1 }\rangle + 2w\langle \phi_{1,1}\rangle + w^2\langle \phi_{0,2}\rangle + w^2 \langle \phi_{1,2}\rangle) \\
&\quad+ (a+1)(1-\delta_{b,0})(b-1)  (2w\langle \phi_{0,1}\rangle + w^2 \langle \phi_{0,2}\rangle) \\
&\quad+ w\langle \phi_{0,1}\delta_{b,0} \rangle ( \psi_{(3,1^{a-2})}  + \psi_{(2^2,1^{a-3})} ).
\end{align*}
To continue, it is necessary to extract coefficients of the form $\left[ w^d \frac{v^{c+d}}{(c+d)!}\right] w^k \langle \phi_{i,j}\rangle.$
From the definition, $\left[\frac{v^{c+d}}{(c+d)!}\right]  \langle \phi_{i,j}\rangle$ is the complete symmetric function $h_{c+d-i-j}$ in $i+j$ indeterminates, with $i$ indeterminates set to 1 and $j$ indeterminates set to $w$. Thus, we have the following.
\begin{lemma}
Let $i,j>0$. Then
\begin{align*}
\left[ w^d \frac{v^{c+d}}{(c+d)!}\right] w^k \langle \phi_{i,j}\rangle &= [w^{d-k}] [t^{c+d-i-j}] (1-t)^{-i}(1-wt)^{-j} \\
&= \begin{cases}
 \binom{c+k-j-1}{c+k-j-i} \binom{d-k+j-1}{d-k} & \text{ if } d\geq k \text{ and } c\geq i+j-k,\\
 0 & \text{ otherwise.}
 \end{cases}
\end{align*}
When $i=0$,
\[
\left[ w^d \frac{v^{c+d}}{(c+d)!}\right] w^k \langle \phi_{0,j}\rangle = \begin{cases} \binom{d-k+j-1}{d-k} & \text{ if } c = j-k, \\
0 &\text{ otherwise,}
\end{cases}
\]
and for $j=0$,
\[
\left[ w^d \frac{v^{c+d}}{(c+d)!}\right] w^k \langle \phi_{i,0}\rangle = \begin{cases} 
\binom{c+d-1}{c+d-i} & \text{ if } d=k, \\
0 & \text{ otherwise.}
\end{cases}
\]
\label{lem:coefficient-extraction}
\end{lemma}
This leads to the following.
\begin{thm}
Suppose $c+d>0$. When $b\geq 1$, the number of $(a,b,c,d)$-dipoles on the torus is given by
\[
 \begin{cases}
(a+1)b(d+1)+ \psi_{3,1^{a-1}} + \psi_{2^2,1^{a-4}}  & \text{ if } c=0, \\
(a+1)(d+1) & \text{ if } c\geq 1.
\end{cases}
\]
When $b=0$, the number of $(a,b,c,d)$-dipoles on the torus is given by
\[
\begin{cases}
\binom{d+2}{d-2} + \psi_{3,1^{a-1}} + \psi_{2^2,1^{a-4}} & \text{ if } c=0, \\
(a+1)(d+1) & \text{ if } c=1, \\
(a+1)c(d+1) &\text { if } c\geq 2.
\end{cases}
\]
When $c=d=0$, then the number of $(a,b,c,d)$-dipoles on the torus is given by Equation (\ref{eqn:ab00-torus-coeff}).
\end{thm}
\begin{proof}
The number of $(a,b,c,d)$-dipoles on the torus is given by
\[
\left[ \frac{x^{a+1}}{(a+1)!}\frac{y^b}{b!}\frac{v^{c+d}}{(c+d)!} w^d\right] \langle \serC^{(1)} \rangle.
\]
The results are obtained by using Lemma \ref{lem:coefficient-extraction} to extract coefficients from the expression for $\serC^{(1)}$ given by Theorem \ref{thm:abcd-torus}, and performing routine simplification.
\end{proof}

\subsection{Solutions for the Double Torus}
\label{subsec:double-torus-solution}

In order to determine the genus 2 solution for $\serB$, the first step is to determine $J'\serB^{(1)}$. 

Using this lemma, the following computation may be done:
\begin{align*}
J'\serB^{(1)} &= x\exp(xf_{(\wdot)}) ( g_{\bdot\bdot \wdot\bdot\bdot}\phi_2(x) + xg_{\bdot\bdot\wdot\bdot\wdot} \phi_1(x) + 2g_{\bdot\bdot\bdot\bdot}f_{(\bdot\wdot)}\phi_3(x) + g_{\bdot\bdot\bdot\bdot\wdot}\phi_2(x)  \\
&\qquad\qquad\qquad+ g_{\bdot\bdot\bdot\wdot\bdot} \phi_2(x) + xg_{\bdot\bdot\bdot\wdot} f_{(\bdot\wdot)} \phi_2(x)) \\
&+ \sum_{n\geq 3} \frac{x^n}{n!}  \psi_{(3,1^{n-3})} (3g_{\bdot\bdot\wdot\wdot\wdot}f_{(\wdot)}^{n-3} + (n-3)g_{\bdot\bdot \wdot} f_{(\wdot\wdot\wdot)} f_{(\wdot)}^{n-4}) \\
&+ \sum_{n\geq 4} \frac{x^n}{n!} \psi_{(2^2,1^{n-4})} (4g_{\bdot\bdot\wdot\wdot} f_{(\wdot\wdot)}f_{(\wdot)}^{n-4} + (n-4) g_{\bdot\bdot\wdot} f_{(\wdot\wdot)}^2 f_{(\wdot)}^{n-5}).
\end{align*}
Proceeding as in the genus 1 case leads to the following.
\begin{lemma}
The generating series for $(a,b,0,0)$-dipoles on the double torus is
\begin{align*}
\serB^{(2)} &= x\exp(xf_{(\wdot)}) ( \tau_{\bdot\wdot\bdot\bdot,2} + x\tau_{\bdot\wdot\bdot\bdot,1} + 2 f_{(\bdot\wdot)} \tau_{\bdot\bdot\bdot,3} + \tau_{\bdot\bdot\bdot\wdot,2} + \tau_{\bdot\bdot\wdot\bdot,2} + xf_{(\bdot\wdot)} \tau_{\bdot\bdot\wdot,2}) \\
&+ 3\tau_{\bdot\wdot\wdot\wdot,0} \sum_{n\geq 3} \frac{x^n}{n!} \psi_{(3,1^{n-3})} f_{(\wdot)}^{n-3} + \tau_{\bdot\wdot,0}f_{(\wdot\wdot\wdot)} \sum_{n\geq 4} \frac{(n-3)x^n}{n!}\psi_{(3,1^{n-3})} f_{(\wdot)}^{n-4} \\
&+ 4\tau_{\bdot\wdot\wdot,0} f_{(\wdot\wdot)} \sum_{n\geq 4} \frac{x^n}{n!} \psi_{(2^2,1^{n-4})} f_{(\wdot)}^{n-4} + \tau_{\bdot\wdot,0} f_{(\wdot\wdot)}^2 \sum_{n\geq 5} \frac{(n-4)x^n}{n!}\psi_{(2^2,1^{n-4})} f_{(\wdot)}^{n-5} \\
&+ [u^4] g_{\bdot} \sum_{D \in \mathcal{D}} \frac{x^{n(D)}}{n(D)!} u^{2g(D)} f_{\lambda'(D)}.
\end{align*}
\label{lemma:ab00-double-torus}
\end{lemma}
The coefficients for $(a,b,0,0)$ dipoles on the double torus may now be obtained as follows.
\begin{cor}
\begin{align*}
\left[\frac{y^b}{b!} \frac{x^{a+1}}{(a+1)!}\right] \langle \serB^{(2)} \rangle &= 3(a+1)\binom{b+2}{5} + a(a+1)\binom{b+2}{4} + 2(a+1)\binom{b+1}{5} \\
&\qquad + a(a+1)\binom{b+1}{4}  + \left(3 \binom{b+2}{3}  + (a-2)b\right) \psi_{(3,1^{a-2})} \\
&\qquad + \left(4 \binom{b+1}{2}  + (a-3) b  \right)\psi_{(2^2,1^{a-3})},
\end{align*}
adopting the convention that $\psi_{(3,1^{a-2})}=0$ when $a<2$, and $\psi_{(2^2,1^{a-3})}=0$ when $a<3$.
\end{cor}
This result is of interest not only for verification purposes, but also as an example, much like Equation (\ref{eqn:ab00-torus-coeff}), of an explicit result that can be obtained in the special case $c=d=0$. Lemma \ref{lemma:ab00-double-torus} can be used to determine the generating series for $(a,b,c,d)$-dipoles on the double torus in the following manner.
\begin{enumerate}
\item
Use Theorem \ref{thm:tau-nonrecursive} and Theorem \ref{thm:sigma-nonrecursive} to write the expressions for $\serB^{(2)}$ (given in Lemma \ref{lemma:ab00-double-torus}) and $\serC^{(1)}$ (given in Theorem \ref{thm:abcd-torus}) in terms of the $g$-type indeterminates.
\item
Determine $C'' \serB^{(2)} + J''\serC^{(1)}$. 
\item
In the expression obtained in the preceding step, replace every instance of $g_{\bdot R} \phi_S$ with $\tau_{R,S}$. 
\item
Set all $f$- and $g$-type indeterminates equal to 1.
\end{enumerate}
The result of applying this process is the expression for $\langle \serC^{(2)}\rangle$ is given in Appendix \ref{sec:appendix-A}. Coefficients may be obtained from this series using the expression for $\tau_{R,S}$ given in Theorem \ref{thm:sigma-nonrecursive}, along with Lemmas \ref{lemma:phi-coefficient-1} and \ref{lem:coefficient-extraction}. Series for surfaces of higher genera may be obtained with the assistance of a computer.

\section{References}

\bibliographystyle{amsplain}
\bibliography{dipole-paper}

\pagebreak
\appendix
\section{Coefficients of $\langle \tau_{R,i,j}\rangle$ in the $(a,b,c,d)$-dipole series $\langle \serC^{(2)}\rangle$ on the double torus}
\label{sec:appendix-A}

 In the tables describing the coefficients of $\langle \serC^{(2)}\rangle$, the following notation is used:
 \[
\begin{array}{rlrl}
D_{3} &:= \sum_{n\geq 3} \psi_{(3,1^{n-3})} \frac{x^n}{n!},
&\quad D_{2,2} &:= \sum_{n\geq 4} \psi_{(2,2,1^{n-4})} \frac{x^n}{n!}, \\
D_3^* &:= \sum_{n\geq 4} (n-3) \psi_{(3,1^{n-3})} \frac{x^n}{n!}, 
&\quad D_{2,2}^* &:= \sum_{n\geq 5} (n-4) \psi_{(2,2,n^{n-4})} \frac{x^n}{n!},
\end{array}
\]
where $\psi_{\lambda}$ is the number of rooted dipoles with face degree sequence $2\lambda$. The coefficients of each $\tau$ function appearing in $\langle \serC^{(2)}\rangle $ are as follows.
\[
\begin{array}{c|l}
\tau\text{-function} & \text{Coefficient in } \langle \serC^{(2)}\rangle\\
\hline
\langle \tau_{\epsilon,0,0} \rangle & wxe^{x} (5\langle \phi_6\rangle + (13+2x)\langle \phi_5\rangle + (11+5x)\langle \phi_4\rangle + (3+4x) \langle \phi_3\rangle + x \langle \phi_2\rangle)\\
& \quad+ 3w(\langle \phi_1\rangle+3\langle \phi_2\rangle + 3\langle \phi_3\rangle + \langle \phi_4 \rangle) D_3+ w(\langle \phi_1\rangle + \langle \phi_2\rangle) (D_3^*+D_{2,2}^*) \\
&\quad+ 4w(\langle \phi_1\rangle + 2\langle \phi_2\rangle + \langle \phi_3\rangle ) D_{2,2} + w[u^4] \sum_{D \in \mathcal{D}} u^{2g(D)} \frac{x^{n(D)}}{n(D)!} \\
\langle \tau_{\wdot,0,0} \rangle & xe^x( \langle \phi_5\rangle + (2+x) \langle \phi_4\rangle +(1+ 2x) \langle \phi_3\rangle)+ 3(\langle \phi_1\rangle +2\langle \phi_2\rangle + \langle \phi_3\rangle) D_3  \\
&\quad + 4(\langle \phi_1\rangle + \langle \phi_2\rangle) D_{2,2} \\
\langle \tau_{\bdot,0,0} \rangle & wxe^{x} (5\langle \phi_6\rangle + (12+2x)\langle \phi_5\rangle + (9+4x)\langle \phi_4\rangle + (2+2x)\langle \phi_3\rangle + x \langle \phi_2\rangle) \\
& \quad+3w(\langle \phi_2\rangle+ 2\langle \phi_3\rangle + \langle \phi_4\rangle) D_3 + w(\langle \phi_1\rangle + \langle \phi_2\rangle) (D_3^*+D_{2,2}^*)\\
& \quad+4w(\langle \phi_2\rangle + \langle \phi_3\rangle) D_{2,2} \\
\langle \tau_{\wdot\wdot,0,0} \rangle & 3(\langle \phi_1\rangle+\langle \phi_2\rangle) D_3 + 4\langle \phi_1\rangle D_{2,2} + D_3^* + D_{2,2}^*\\
\langle \tau_{\wdot\wdot,0,1} \rangle & wx^2e^x( \langle \phi_1\rangle + \langle \phi_2\rangle)+ w(D_3^*+ D_{2,2}^*) \\
\langle \tau_{\wdot\wdot,0,2} \rangle & w^2x^2e^x (\langle \phi_1\rangle+ \langle \phi_2\rangle) \\
\langle \tau_{\wdot\wdot,0,3} \rangle & w^3xe^x \\
\langle \tau_{\wdot\wdot,0,4} \rangle & 2w^4xe^x \\
\langle \tau_{\wdot\wdot,0,5} \rangle & w^5xe^x \\
\langle \tau_{\wdot\wdot,1,1} \rangle & wx^2e^x(1 + \langle \phi_1\rangle) \\
\langle \tau_{\wdot\wdot,1,2} \rangle & w^2x^2e^x(1+ \langle \phi_1\rangle) \\
\langle \tau_{\wdot\wdot,2,1} \rangle & wxe^x(1+ \langle \phi_1\rangle) + wx^2e^x \\
\langle \tau_{\wdot\wdot,2,2} \rangle & w^2xe^x(1+ \langle \phi_1\rangle) + w^2x^2e^x \\
\langle \tau_{\wdot\wdot,3,1} \rangle & 2wxe^x \\
\langle \tau_{\wdot\wdot,3,2} \rangle & 2w^2xe^x \\
\langle \tau_{\wdot\bdot,0,0} \rangle & xe^x (\langle \phi_4\rangle + (1+x)\langle \phi_3\rangle) + \langle \phi_1\rangle (D_3^* + D_{2,2}^*)\\
\langle \tau_{\bdot\wdot,0,0} \rangle & wxe^x (\langle \phi_5\rangle + (2+x) \langle \phi_4\rangle + (1+2x) \langle \phi_3\rangle) + 3w(\langle \phi_2\rangle+ \langle \phi_3\rangle) D_3 \\
& \quad+ 4w\langle \phi_2 \rangle D_{2,2} \\
\langle \tau_{\bdot\bdot,0,0} \rangle & wxe^{x} (4\langle \phi_5\rangle + (5+x) \langle \phi_4\rangle + \langle \phi_3\rangle + x \langle \phi_2\rangle) \\
\langle \tau_{\bdot\bdot,0,1} \rangle & w^2xe^x (\langle \phi_2\rangle+2\langle \phi_3\rangle) + w^2(D_3+ D_{2,2}) \\
\langle \tau_{\bdot\bdot,0,2} \rangle & 2w^3xe^x(\langle \phi_1\rangle + 2\langle \phi_2 \rangle+2\langle \phi_3\rangle)+ w^3 (D_3+ D_{2,2})  \\
\langle \tau_{\bdot\bdot,0,3} \rangle & 2w^4xe^x (\langle \phi_1\rangle+ \langle \phi_2\rangle) \\
\langle \tau_{\bdot\bdot,0,4} \rangle & 4w^5(e^x-1) \\
\langle \tau_{\bdot\bdot,0,5} \rangle & 9w^6(e^x-1) \\
\langle \tau_{\bdot\bdot,0,6} \rangle & 5w^7(e^x-1) \\
\langle \tau_{\bdot\bdot,1,1} \rangle & w^2xe^x \langle \phi_2\rangle \\
\langle \tau_{\bdot\bdot,1,2} \rangle & (2+ 2\langle \phi_1\rangle + \langle \phi_2\rangle )w^3xe^x \\
\langle \tau_{\bdot\bdot,1,3} \rangle & 2w^4xe^x(1 + \langle \phi_1\rangle) \\
\langle \tau_{\bdot\bdot,2,2} \rangle & 2w^3xe^x \\
\langle \tau_{\bdot\bdot,2,3} \rangle & 2w^4xe^x 
\end{array}
\]

\pagebreak

\[
\begin{array}{cc}
\begin{array}{c|l}
\tau\text{-function} & \text{Coefficient in } \langle \serC^{(2)}\rangle\\
\hline
\langle \tau_{\wdot\wdot\wdot,0,0} \rangle & 3w\langle \phi_1\rangle D_3 + 4D_{2,2} \\
\langle \tau_{\wdot\wdot\wdot,0,1} \rangle & 4wD_{2,2} \\
\langle \tau_{\wdot\wdot\wdot,1,0} \rangle & x^2e^x \langle \phi_1\rangle \\
\langle \tau_{\wdot\wdot\wdot,1,1} \rangle & wx^2e^x \langle \phi_1 \rangle \\
\langle \tau_{\wdot\wdot\wdot,2,0} \rangle & x^2e^x + xe^x \langle \phi_1\rangle \\
\langle \tau_{\wdot\wdot\wdot,2,1} \rangle & wx^2e^x + (2+ \langle \phi_1\rangle) wxe^x\\
\langle \tau_{\wdot\wdot\wdot,2,2} \rangle & 2w^2xe^x \\
\langle \tau_{\wdot\wdot\wdot,3,0} \rangle & 2xe^x \\
\langle \tau_{\wdot\wdot\wdot,3,1} \rangle & 2wxe^x \\
\langle \tau_{\wdot\wdot\bdot,0,0} \rangle & x^2e^x \langle \phi_2\rangle \\
\langle \tau_{\wdot\wdot\bdot,0,1} \rangle & wx^2e^x \langle \phi_2\rangle \\
\langle \tau_{\wdot\wdot\bdot,0,3} \rangle & w^3xe^x \\
\langle \tau_{\wdot\wdot\bdot,0,4} \rangle & w^4xe^x \\
\langle \tau_{\wdot\bdot\wdot,0,1} \rangle & wxe^x\langle \phi_2\rangle \\
\langle \tau_{\wdot\bdot\wdot,0,2} \rangle & w^2xe^x \langle \phi_2\rangle \\
\langle \tau_{\wdot\bdot\wdot,1,1} \rangle & 2wxe^x\langle \phi_1\rangle \\
\langle \tau_{\wdot\bdot\wdot,1,2} \rangle & w^2xe^x \langle \phi_1\rangle \\
\langle \tau_{\wdot\bdot\wdot,2,1} \rangle & wxe^x \\
\langle \tau_{\wdot\bdot\wdot,2,2} \rangle & w^2xe^x \\
\langle \tau_{\bdot\wdot\wdot,0,0} \rangle& 3w\langle \phi_2\rangle D_3 + 4w\langle \phi_1\rangle D_{2,2}\\
\langle \tau_{\wdot\bdot\bdot,0,0} \rangle &  xe^{x} (\langle \phi_3\rangle + x \langle \phi_2\rangle) \\
\langle \tau_{\bdot\wdot\bdot,0,0} \rangle & wxe^x (\langle \phi_4\rangle +(1+ x)\langle \phi_3\rangle) \\
\langle \tau_{\bdot\wdot\bdot,0,1} \rangle & w^2xe^x \langle \phi_2\rangle \\
\langle \tau_{\bdot\wdot\bdot,0,2} \rangle & w^3xe^x \langle \phi_2\rangle \\
\langle \tau_{\bdot\wdot\bdot,1,1} \rangle & 2w^2xe^x \langle \phi_1\rangle \\
\langle \tau_{\bdot\wdot\bdot,1,2} \rangle & 2w^3xe^x\langle \phi_1\rangle \\
\langle \tau_{\bdot\wdot\bdot,2,1} \rangle & w^2xe^x \\
\langle \tau_{\bdot\wdot\bdot,2,2} \rangle & w^3xe^x \\
\langle \tau_{\bdot\bdot\wdot,0,0} \rangle & wxe^x (\langle \phi_4\rangle + (1+x)\langle \phi_3 \rangle) \\
\langle \tau_{\bdot\bdot\wdot,1,0} \rangle & wxe^x\langle \phi_2 \rangle \\
\langle \tau_{\bdot\bdot\wdot,1,1} \rangle & w^2 xe^x (\langle \phi_1 \rangle + \langle \phi_2 \rangle) \\
\langle \tau_{\bdot\bdot\wdot,1,2} \rangle & w^3 xe^x \langle \phi_1 \rangle \\
\langle \tau_{\bdot\bdot\wdot,2,1} \rangle & w^2xe^x \\
\langle \tau_{\bdot\bdot\wdot,2,2} \rangle & w^3xe^x \\
\langle \tau_{\bdot\bdot\bdot,0,0} \rangle & wxe^x (2\langle \phi_3\rangle + 3\langle \phi_4 \rangle) \\
\langle \tau_{\bdot\bdot\bdot,0,1} \rangle & w^2xe^x( 3\langle \phi_2\rangle+ 2\langle \phi_3\rangle) \\
\langle \tau_{\bdot\bdot\bdot,0,2} \rangle & 3w^3xe^x \langle \phi_2\rangle \\
\langle \tau_{\bdot\bdot\bdot,0,4} \rangle & 10w^5(e^x-1) \\
\langle \tau_{\bdot\bdot\bdot,0,5} \rangle & 8w^6(e^x-1) 
\end{array} 

& 

\begin{array}{c|l}
\tau\text{-function} & \text{Coefficient in } \langle \serC^{(2)}\rangle\\
\hline
\langle \tau_{\wdot\wdot\wdot\wdot,0,0} \rangle & 3D_3 \\
\langle \tau_{\wdot\wdot\wdot\wdot,0,1} \rangle & 3wD_3 \\
\langle \tau_{\wdot\wdot\wdot\wdot,1,0} \rangle & x^2e^x \\
\langle \tau_{\wdot\wdot\wdot\wdot,2,0} \rangle & 3xe^x \\
\langle \tau_{\wdot\wdot\wdot\wdot,1,1} \rangle & wx^2 e^x \\
\langle \tau_{\wdot\wdot\wdot\wdot,2,1} \rangle & 3wxe^x \\
\langle \tau_{\wdot\wdot\bdot\wdot,0,0} \rangle & x^2e^x\langle \phi_1\rangle \\
\langle \tau_{\wdot\wdot\bdot\wdot,0,1} \rangle & wx^2e^x \langle \phi_1\rangle \\
\langle \tau_{\wdot\wdot\bdot\wdot,1,1} \rangle & wxe^x \\
\langle \tau_{\wdot\wdot\bdot\wdot,1,2} \rangle & w^2xe^x \\
\langle \tau_{\wdot\wdot\bdot\bdot,0,2} \rangle & xw^2e^x \\
\langle \tau_{\wdot\wdot\bdot\bdot,0,3} \rangle & w^3xe^x \\
\langle \tau_{\wdot\bdot\wdot\wdot,1,0} \rangle & xe^x \langle \phi_1\rangle \\
\langle \tau_{\wdot\bdot\wdot\wdot,1,1} \rangle & wxe^x(1 + \langle \phi_1\rangle ) \\
\langle \tau_{\wdot\bdot\wdot\wdot,1,2} \rangle & w^2xe^x \\
\langle \tau_{\wdot\bdot\wdot\bdot,0,0} \rangle & xe^x\langle \phi_2\rangle \\
\langle \tau_{\wdot\bdot\wdot\bdot,0,1} \rangle & wxe^x \langle \phi_2\rangle \\
\langle \tau_{\wdot\bdot\bdot\wdot,0,1} \rangle & wxe^x\langle \phi_1\rangle \\
\langle \tau_{\wdot\bdot\bdot\wdot,0,2} \rangle & w^2xe^x \langle \phi_1\rangle \\
\langle \tau_{\bdot\wdot\wdot\wdot,0,0} \rangle & 3\langle \phi_1 \rangle D_3 \\
\langle \tau_{\bdot\wdot\wdot\bdot,1,1} \rangle & w^2xe^x \\
\langle \tau_{\bdot\wdot\wdot\bdot,1,2} \rangle & w^3xe^x \\
\langle \tau_{\bdot\wdot\bdot\wdot,1,0} \rangle & wxe^x \langle \phi_1\rangle \\
\langle \tau_{\bdot\wdot\bdot\wdot,1,1} \rangle & w^2xe^x \langle \phi_1\rangle \\
 \langle \tau_{\bdot\wdot\bdot\bdot,0,0} \rangle & wxe^{x}(\langle \phi_3 \rangle + (1+x)\langle \phi_2\rangle) \\
 \langle \tau_{\bdot\wdot\bdot\bdot,0,1} \rangle & w^2xe^x (\langle \phi_1 \rangle +\langle \phi_2\rangle) \\
 \langle \tau_{\bdot\wdot\bdot\bdot,0,2} \rangle & w^3 xe^x \langle \phi_1\rangle \\
 \langle \tau_{\bdot\bdot\wdot\wdot,1,1} \rangle & w^2xe^x \\
 \langle \tau_{\bdot\bdot\wdot\wdot,1,2} \rangle & w^3 xe^x \\
 \langle \tau_{\bdot\bdot\wdot\bdot,0,0} \rangle & wxe^x (\langle \phi_2 \rangle + \langle \phi_3 \rangle) \\
 \langle \tau_{\bdot\bdot\wdot\bdot,0,1} \rangle & w^2xe^x \langle \phi_1 \rangle \\
 \langle \tau_{\bdot\bdot\wdot\bdot,0,2} \rangle & w^3 xe^x \langle \phi_1\rangle \\
 \langle \tau_{\bdot\bdot\bdot\wdot,0,0} \rangle & wxe^x \langle \phi_3 \rangle \\
 \langle \tau_{\bdot\bdot\bdot\wdot,0,1} \rangle & w^2xe^x(\langle \phi_1\rangle + \langle \phi_2\rangle) \\
 \langle \tau_{\bdot\bdot\bdot\wdot,0,2} \rangle & w^3xe^x \langle \phi_1\rangle \\
 \langle \tau_{\bdot\bdot\bdot\bdot,0,3} \rangle & 8w^4(e^x-1)  \\
 \langle \tau_{\bdot\bdot\bdot\bdot,0,4} \rangle & 6w^5(e^x-1) 
 \end{array}
 
 \end{array}
 \]

\end{document}